%% file: ex_article.tex
\documentclass[review,onefignum,onetabnum]{siamonline171218}
\usepackage{pst-node, lipsum}
\usepackage{tikz-cd}
\usepackage{changepage}

\newcommand{\so}{\mathfrak{so}}

\newcommand{\g}{\mathfrak{g}}

\newcommand{\Sg}{\mathcal{S}}
\newcommand{\Tg}{\mathcal{T}}
\newcommand{\Yg}{\mathcal{Y}}
\newcommand{\Zg}{\mathcal{Z}}
\newcommand{\Hor}{\mathcal{H}}

\newcommand{\Exp}{\text{Exp}}
\newcommand{\Log}{\text{Log}}
\newcommand{\ext}{\text{ext}}
\newcommand{\Bi}{\text{Bi}}
\newcommand{\HorLC}{\tilde{\nabla}^{\Hor}}
\newcommand{\HorD}{\tilde{D}^{\Hor}}
\newcommand{\hcon}{\tilde{\nabla}^{\mathfrak{h}}}

\newcommand{\N}{\mathbb{N}}

\newcommand{\R}{\mathbb{R}}
\newcommand{\grad}{\text{grad}}

\newcommand{\tr}{\text{tr}}
\newcommand{\SO}{\text{SO}}

\newcommand{\ad}{\text{ad}}
\newcommand{\Ad}{\text{Ad}}
\newcommand{\I}{\mathbb{I}}
\newcommand{\todo}[1]{
  \vspace{5mm}\par \noindent
  \marginpar{\textsc{ToDo}}
  \framebox{
    \begin{minipage}[c]{0.95\linewidth}
      \tt #1
    \end{minipage}
  }
  \vspace{5mm}\par
}
\input{ex_shared}

\ifpdf
\hypersetup{
  pdftitle={Reduction by Symmetry in Obstacle Avoidance Problems on Riemannian Manifolds},
  pdfauthor={J. R. Goodman and L. J. Colombo}
}
\fi

\myexternaldocument{ex_supplement}

\begin{document}

\maketitle

% REQUIRED
\begin{abstract}
 This paper studies the reduction by symmetry of a variational obstacle avoidance problem. We derive the reduced necessary conditions in the case of Lie groups endowed with a left-invariant metric, and for its corresponding Riemannian homogeneous spaces by considering an alternative variational problem written in terms of a connection on the horizontal bundle of the Lie group. A number of special cases where the obstacle avoidance potential can be computed explicitly are studied in detail, and these ideas are applied to the obstacle avoidance task for a rigid body evolving on $\SO(3)$ and for the unit sphere $S^2$.
\end{abstract}

% REQUIRED

% REQUIRED

\section{Introduction}
In the last decades, energy-optimal path planning on nonlinear spaces such as Riemannian manifolds has emerged as an active field of interest due to its numerous applications in manufacturing, imaging, and robotics \cite{bonard, hussein}. Most frequently, one typically wishes to construct paths connecting some set of knot points—interpolating some set of given positions velocities. Though it is sometimes desirable to interpolate higher order velocities as well. The use of variationally defined curves has a rich history due to the regularity and optimal nature of the solutions. In particular, the so-called Riemannian {cubic} splines are a particularly pervasive interpolant, which themselves are composed of Riemannian {cubic} polynomials—curves which minimize the total squared (covariant) acceleration among all sufficiently regular curves satisfying some boundary conditions in positions and velocities—that are subsequently glued together to create a spline. Riemannian cubic polynomials carry have deep connections with Riemannian geometry, which often parallels the theory of geodesics. This has been studied extensively in the literature (see \cite{noakes, CroSil:95, Giambo, Margarida} for a detailed account of Riemannian cubics and \cite{RiemannianPoly, marg} for some results with higher-order Riemannian polynomials).

In practical applications, it is often the case that there are obstacles or regions in space which need to be avoided. In this case, the action functional may be augmented with an artificial potential term which grows inversely to the Riemannian distance from the obstacle. This strategy was employed for instance in \cite{BlCaCoCDC, BlCaCoIJC}, where necessary conditions for extrema in obstacle avoidance problems on Riemannian manifolds were derived, and applications to interpolation problems on manifolds and to energy-minimum problems on Lie groups and symmetric spaces endowed with a bi-invariant metric were studied. Similar strategies have been implemented for collision avoidance problems for multi-agent systems evolving on Riemannian problems, as in \cite{mishal}. Existence of global minimizers and safety guarantees were studied in \cite{goodman2021obstacle} for the obstacle avoidance problem, and \cite{goodman2022collision} for the collision avoidance problem. Sufficient conditions for optimality in obstacle avoidance problems—analogous to the theory of Jacobi Fields for geodesics—was studied in \cite{goodman2022sufficient}. 

The major drawback of utilizing variationally defined curves in the task of path planning is that they are frequently impossible to compute explicitly, and often computationally expensive to approximate numerically. Fortunately, many applications evolve naturally on Lie groups (such as $\SO(3)$ and $\text{SE}(3)$) and homogeneous spaces (such as $S^2$). Such spaces have special structures that allow the variational principles (and their resulting necessary conditions for optimality, often called the Euler-Lagrange equations) to be reduced. In particular, the Euler-Lagrange equations reduce to some system of equations, often referred to as the Euler-Poincar\'e equations, which evolve on a vector space. Hence, the reduced equations are well suited to numerical integration schemes, from which the corresponding solution to the Euler-Lagrange equations can be obtained via a reconstruction equation. Reduction by symmetry for Riemannian cubics has been studied in \cite{Altafini}.

The inclusion of potentials in the action functional complicates the reduction process. First, the potential must admit certain symmetries to allow for reduction. In the literature, this typically involves the existence of a space of parameters on which the underlying Lie group acts, and a parameter dependent extension of the potential which is invariant under some Lie group action and agrees with the original potential for some particular choice of parameter. Such problems are typically referred to as reduction with symmetry breaking. In the case that the underlying Riemannian manifold is a Lie group, the problem of reduction with symmetry breaking has been studied in \cite{gay2012invariant} for variational principles of arbitrary order and where the parameter space is a manifold. For homogeneous spaces, reduction with symmetry breaking has only been studied for variational principles of order $1$ \cite{vizman}. Reduction with symmetry breaking for the particular case of obstacle avoidance potential has not been addressed in the literature, though a related first-order Optimal control problem with symmetry breaking has been studied in \cite{OptConSym}. A particular challenge for reduction with symmetry breaking in the obstacle avoidance problem comes from practical considerations. Namely, the Riemannian distance function is difficult—if not impossible—to compute explicitly outside of certain special cases (such as Lie groups with bi-invariant metrics and Riemannian symmetric spaces), and many applications (rigid body motions, for example) do not possess such structures. Even numerical approximations are very costly, rendering the strategy undesirable for the purpose of path-planning. The primary motivation for this paper is to extend the notion of reduction with symmetry breaking to the obstacle avoidance problem, and to do so in such a way that the generated trajectories may be found in a cost-effective manner—amplifying considerably the class of possible applications. 

The main contributions of this paper are as follows. (i) In Proposition \ref{prop: reduction_left_inv}, necessary conditions for reduction with symmetry breaking for the obstacle avoidance task are derived in the case of a Lie group endowed with a left-invariant metric. (ii) In Corollary \ref{cor: reduction_bi_inv}, these results are specialized to the case of a Lie group endowed with a left-invariant metric. It is shown how the bi-invariant metric allows the Riemannian distance to be computed explicitly. Then (iii), in Corollary \ref{prop: reduction_left_inv_comp}, another particular case is considered where the Lie group is endowed with both a left-invariant (from which dynamics are derived) and bi-invariant metric (with respect to which the Riemannian distance is defined). In particular, this allows to compute the Riemannian distance (and thus the potential) explicitly even in the case where the necessary conditions are expressed with respect to a left-invariant (but not bi-invariant) metric. (iv) In Section \ref{section: rigid_body_cubics}, these principles are applied to reduce the necessary conditions for optimality corresponding to a rigid body evolving on $\SO(3)$ with an obstacle avoidance potential. In Section \ref{Sec: Reduction_Homo}, we then turn our attention to homogeneous spaces of Lie groups endowed with left-invariant metrics. (v) In Proposition \ref{thm: necessary_homo_G}, we derive necessary conditions for optimality corresponding to the obstacle avoidance task, written in terms of a connection on the horizontal bundle. This formalism differs from the literature, which considers the Levi-Civita connection. (vi) In Proposition \ref{prop: reduction_left_inv_H}, we then reduce the necessary conditions for optimality with symmetry breaking. We then consider a number of particular cases in Section \ref{Sec: homo_bi}. In particular, (vii) in Proposition \ref{prop: reduction_bi_inv_H}, we consider the reduced necessary conditions for optimality with symmetry breaking for homogeneous spaces of Lie groups with bi-invariant metrics, and of Lie groups equipped with both left-invariant and bi-invariant metrics. (viii) In Proposition \ref{prop: reduction_sym_left}, we find the reduced necessary conditions for optimality with symmetry breaking for Riemannian symmetric spaces with left-invariant, bi-invariant metrics, and both left-invariant and bi-invariant metrics. Finally, (ix) we study the particular case of obstacle avoidance on homogeneous spaces, and in Section \ref{Section: S^2} we apply the results to the case of the unit sphere $S^2$ endowed with the round metric.

The paper is organized as follows. In Section \ref{Sec: background}, we provide the necessary background for Riemannian manifolds, and in Section \ref{sec: background_Lie} we specialize the discussion to left-invariant Riemannian metrics on Lie groups. In Theorem \ref{thm: EP_geo}, the Euler-Poincar\'e equations for a geodesic are derived using our formalism, and in Section \ref{section: example_geo_SO3}, they are applied to the special case of a rigid body evolving on $\SO(3)$. In section \ref{Sec: Necessary conditions}, we discuss the general framework for the variational obstacle avoidance problem, and in Section \ref{reduction_obs_avoid_G}, we reduce the necessary conditions for optimality with the symmetry breaking obstacle avoidance potential on Lie groups endowed with left-invariant metrics. In Section \ref{sec: red_biinv}, we then specialize the discussion to Lie groups which admit bi-invariant metrics. In Section \ref{section: rigid_body_cubics}, we return to the example of a rigid body evolving on $\SO(3)$, where we study the corresponding obstacle avoidance problem. Section \ref{Sec: Reduction_Homo} regards reduction with symmetry breaking on Riemannian homogeneous spaces, where the underlying Lie group is endowed with a left-invariant metric. In Section \ref{Sec: homo_bi}, we consider the special cases where the Lie group admits a bi-invariant metric, and where the Riemannian homogeneous space is in fact a Riemannian symmetric space. Finally, Section \ref{Section: S^2} applies these principles to obstacle avoidance on the unit sphere $S^2$.

\section{Riemannian Manifolds}\label{Sec: background}
Let $(Q, \left< \cdot, \cdot\right>)$ be an $n$-dimensional \textit{Riemannian
manifold}, \newline where $Q$ is an $n$-dimensional smooth manifold and $\left< \cdot, \cdot \right>$ is a positive-definite symmetric covariant 2-tensor field called the \textit{Riemannian metric}. That is, to each point $q\in Q$ we assign a positive-definite inner product $\left<\cdot, \cdot\right>_q:T_qQ\times T_qQ\to\mathbb{R}$, where $T_qQ$ is the \textit{tangent space} of $Q$ at $q$ and $\left<\cdot, \cdot\right>_q$ varies smoothly with respect to $q$. The length of a tangent vector is determined by its norm, defined by
$\|v_q\|=\left<v_q,v_q\right>^{1/2}$ with $v_q\in T_qQ$. For any $p \in Q$, the Riemannian metric induces an invertible map $\cdot^\flat: T_p Q \to T_p^\ast Q$, called the \textit{flat map}, defined by $X^\flat(Y) = \left<X, Y\right>$ for all $X, Y \in T_p Q$. The inverse map $\cdot^\sharp: T_p^\ast Q \to T_p Q$, called the \textit{sharp map}, is similarly defined implicitly by the relation $\left<\alpha^\sharp, Y\right> = \alpha(Y)$ for all $\alpha \in T_p^\ast Q$. Let $C^{\infty}(Q)$ and $\Gamma(TQ)$ denote the spaces of smooth scalar fields and smooth vector fields on $Q$, respectively. The sharp map provides a map from $C^{\infty}(Q) \to \Gamma(TQ)$ via $\grad f(p) = df_p^\sharp$ for all $p \in Q$, where $\grad f$ is called the \textit{gradient vector field} of $f \in C^{\infty}(Q)$. More generally, given a map $V: Q \times \cdots \times Q \to \R$ (with $m$ copies of $Q$), we may consider the gradient vector field of $V$ with respect to $i^{\text{th}}$ component as $\grad_i V(q_1, \dots, q_m) = \grad U(q_i)$, where $U(q) = V(q_1, \dots, q_{i-1}, q, q_{i+1}, \dots, q_m)$ for all $q, q_1, \dots, q_m \in Q$.

Vector fields are a special case of smooth sections of vector bundles. In particular, given a vector bundle $(E, Q, \pi)$ with total space $E$, base space $Q$, and projection $\pi: E \to Q$, where $E$ and $Q$ are smooth manifolds, a \textit{smooth section} is a smooth map $X: Q \to E$ such that $\pi \circ X = \text{id}_Q$, the identity function on $Q$. We similarly denote the space of smooth sections on $(E, Q, \pi)$ by $\Gamma(E)$. A \textit{connection} on $(E, Q, \pi)$ is a map $\nabla: \Gamma(TQ) \times \Gamma(E) \to \Gamma(TQ)$ which is $C^{\infty}(Q)$-linear in the first argument, $\R$-linear in the second argument, and satisfies the product rule $\nabla_X (fY) = X(f) Y + f \nabla_X Y$ for all $f \in C^{\infty}(Q), \ X \in \Gamma(TQ), \ Y \in \Gamma(E)$. The connection plays a role similar to that of the directional derivative in classical real analysis. The operator
$\nabla_{X}$ which assigns to every smooth section $Y$ the vector field $\nabla_{X}Y$ is called the \textit{covariant derivative} (of $Y$) \textit{with respect to $X$}. Connections induces a number of important structures on $Q$, a particularly ubiquitous such structure is the \textit{curvature endomorphism}, which is a map $R: \Gamma(TQ) \times \Gamma(TQ) \times \Gamma(E) \to \Gamma(TQ)$ defined by $R(X,Y)Z := \nabla_{X}\nabla_{Y}Z-\nabla_{Y}\nabla_{X}Z-\nabla_{[X,Y]}Z$ for all $X, Y \in \Gamma(TQ), \ Z \in \Gamma(E)$. Qualitatively, the curvature endomorphism measures the extent to which covariant derivatives commute with one another. We further define the \textit{curvature tensor} $\text{Rm}$ on $Q$ via $\text{Rm}(X, Y, Z, W) := \left<R(X, Y)Z, W\right>$.

We now specialize our attention to \textit{affine connections}, which are connections on $TQ$. Let $q: I \to Q$ be a smooth curve parameterized by $t \in I \subset \R$, and denote the set of smooth vector fields along $q$ by $\Gamma(q)$. Then for any affine connection $\nabla$ on $Q$, there exists a unique operator $D_t: \Gamma(q) \to \Gamma(q)$ (called the \textit{covariant derivative along $q$}) which agrees with the covariant derivative $\nabla_{\dot{q}}\tilde{W}$ for any extension $\tilde{W}$ of $W$ to $Q$. A vector field $X \in \Gamma(q)$ is said to be \textit{parallel along $q$} if $\displaystyle{D_t X\equiv 0}$. For $k\in \N$, the $k$th-order covariant derivative of $W$ along $q$, denoted by $\displaystyle{D_t^k W}$, can then be inductively defined by $\displaystyle{D_t^k W = D_t \left(D_t^{k-1} W\right)}$. 

The covariant derivative allows to define a particularly important family of smooth curves on $Q$ called \textit{geodesics}, which are defined as the smooth curves $\gamma$ satisfying $D_t \dot{\gamma} = 0$. Moreover, geodesics induce a map $\mathrm{exp}_q:T_qQ\to Q$ called the \textit{exponential map} defined by $\mathrm{exp}_q(v) = \gamma(1)$, where $\gamma$ is the unique geodesic verifying $\gamma(0) = q$ and $\dot{\gamma}(0) = v$. In particular, $\mathrm{exp}_q$ is a diffeomorphism from some star-shaped neighborhood of $0 \in T_q Q$ to a convex open neighborhood $\mathcal{B}$ (called a \textit{goedesically convex neighborhood}) of $q \in Q$. It is well-known that the Riemannian metric induces a unique torsion-free and metric compatible connection called the \textit{Riemannian connection}, or the \textit{Levi-Civita connection}. Along the remainder of this paper, we will assume that $\nabla$ is the Riemannian connection. For additional information on connections and curvature, we refer the reader to \cite{Boothby,Milnor}. When the covariant derivative $D_t$ corresponds to the Levi-Civita connection, geodesics can also be characterized as the critical points of the length functional $\displaystyle{L(\gamma) = \int_0^1 \|\dot{\gamma}\|dt}$ among all unit-speed \textit{piece-wise regular} curves $\gamma: [a, b] \to Q$ (that is, where there exists a subdivision of $[a, b]$ such that $\gamma$ is smooth and satisfies $\dot{\gamma} \ne 0$ on each subdivision). Equivalently, we may characterize geodesics by the critical points of the \textit{energy functional} $\displaystyle{\mathcal{E} = \frac12 \int_a^b \|\dot{\gamma}\|^2 dt}$ among all $C^1$ piece-wise smooth curves $\gamma: [a, b] \to Q$ parameterized by arc-length. The length functional induces a metric $d: Q \times Q \to \R$ called the \textit{Riemannian distance} via $d(p, q) = \inf\{L(\gamma): \ \gamma \ \text{regular, } \gamma(a) = p, \ \gamma(b) = q\}$.

If we assume that $Q$ is \textit{complete} (that is, $(Q, d)$ is a complete metric space), then by the Hopf-Rinow theorem, any two points $x$ and $y$ in $Q$ can be connected by a (not necessarily unique) minimal-length geodesic $\gamma_{x,y}$. In this case, the Riemannian distance between $x$ and $y$ can be defined by $\displaystyle{d(x,y)=\int_{0}^{1}\Big{\|}\frac{d \gamma_{x,y}}{d s}(s)\Big{\|}\, ds}$. Moreover, if $y$ is contained in a geodesically convex neighborhood of $x$, we can write the Riemannian distance by means of the Riemannian exponential as $d(x,y)=\|\mbox{exp}_x^{-1}y\|.$ 

Given $\xi = (q_a, v_a), \eta = (q_b, v_b) \in TQ$ with $a < b \in \R$, we will say that a curve $\gamma: [a, b] \to Q$ is \textit{admissible} if it is $C^1$, piece-wise smooth, and satisfied $\gamma(a) = q_a, \ \gamma(b) = q_b, \ \dot{\gamma}(a) = v_a, \ \dot{\gamma}(b) = v_b$ by $\Omega_{\xi, \eta}^{a, b}$. The space of admissible curves will be denoted by $\Omega$, and it has the structure of a smooth infinite-dimensional manifold. Its tangent space $T_x \Omega$ consists of all $C^1$, piece-wise smooth vector fields $X$ along $x$ satisfying $X(a) = X(b) = D_tX(a) = D_tX(b) = 0$.

An \textit{admissible variation} of $x$ is a family of curves $\Gamma: (-\epsilon, \epsilon) \times [a, b] \to Q$ such that:
\begin{enumerate}
    \item $\Gamma(r, \cdot) \in \Omega$ for all $r \in (-\epsilon, \epsilon)$,
    \item $\Gamma(0, t) = x(t)$ for all $t \in [a, b]$.
\end{enumerate}
The variational vector field corresponding to $\Gamma$ is defined by $\partial_s \Gamma(0, t)$. An \textit{admissible proper variation} of $x$ is an admissible variation such that 
$\Gamma(s, a) = x(a)$ and $\Gamma(s, b) = x(b)$ for all $s \in (-\epsilon, \epsilon)$. It can be seen that any $X \in T_x \Omega$ is the variational vector field of some admissible proper variation, and conversely, the variational vector field on any admissible proper variation belongs to $T_x \Omega$. 

\subsection{Riemannian geometry on Lie Groups}\label{sec: background_Lie}

Let $G$ be a Lie group with Lie algebra $\g := T_{e} G$, where $e$ is the identity element of $G$. The left-translation map $L: G \times G \to G$ provides a group action of $G$ on itself under the relation $L_{g}h := gh$ for all $g, h \in G$. Given any inner-product $\left< \cdot, \cdot \right>_{\g}$ on $\g$, left-translation provides us with a Riemannian metric $\left< \cdot, \cdot \right>$ on $G$ via the relation:
\begin{align*}
    \left< X_g, Y_g \right> := \left< L_{g^{-1}\ast} X_g, L_{g^{-1}\ast} Y_g \right>_{\g},
\end{align*}
for all $g \in G, X_g, Y_g \in T_g G$, where the notation $L_{g^{\ast}}$ stands for the push-forward of $L_g$, which is well-defined because $L_g: G \to G$ is a diffeomorphism for all $g \in G$. Such a Riemannian metric is called \textit{left-invariant}, and it follows immediately that there is a one-to-one correspondence between left-invariant Riemannian metrics on $G$ and inner products on $\g$, and that $L_g: G \to G$ is an isometry for all $g \in G$ by construction. Any Lie group equipped with a left-invariant metric is complete as a Riemannian manifold. In the remainder of the section, we assume that $G$ is equipped with a left-invariant Riemannian metric.

We call a vector field $X$ on $G$ \textit{left-invariant} if $L_{g\ast} X = X$ for all $g \in G$, and we denote the set of all left-invariant vector fields on $G$ by $\mathfrak{X}_L(G)$. It is well-known that the map $\phi: \g \to \mathfrak{X}_L(G)$ defined by $\phi(\xi)(g) = L_{g\ast} \xi$ for all $\xi \in \g, g \in G$ is an isomorphism between vector spaces. This isomorphism allows us to construct an operator $\nabla^{\g}: \g \times \g \to \g$ defined by:
\begin{align}
    \nabla^{\g}_{\xi} \eta := \nabla_{\phi(\xi)} \phi(\eta)(e),\label{g-connection}
\end{align}
for all $\xi, \eta \in \g$, where $\nabla$ is the Levi-Civita connection on $G$ corresponding to the left-invariant Riemannian metric $\left< \cdot, \cdot \right>$. Although $\nabla^{\g}$ is not a connection, we shall refer to it as the \textit{Riemannian $\g$-connection} corresponding to $\nabla$ because of the similar properties that it satisfies:
\begin{lemma}\label{lemma: covg_prop}
$\nabla^\g: \g \times \g \to \g$ is $\R$-bilinear, and for all $\xi, \eta, \sigma \in \g$, the following relations hold:
\begin{enumerate}
    \item $\nabla_{\xi}^{\g} \eta - \nabla_{\eta}^{\g} \xi = \left[ \xi, \eta \right]_{\g}$,
    \item $\left< \nabla_{\sigma}^{\g} \xi, \eta \right> + \left<  \xi, \nabla_{\sigma}^{\g}\eta \right> = 0$.
\end{enumerate}
\end{lemma}
\begin{proof}
The $\R$-bilinearity of $\nabla^\g$ follows immediately from \eqref{g-connection}. For $1.$, observe that: 
\begin{align*}
    \nabla_{\xi}^{\g} \eta - \nabla_{\eta}^{\g} \xi &= \left(\nabla_{\phi(\xi)} \phi(\eta) - \nabla_{\phi(\eta)} \phi(\xi)\right)(e) \\
    &= \left[\phi(\xi), \phi(\eta)\right](e) \\
    &= \left[\xi, \eta\right]_{\g}
\end{align*}
since $\nabla$ is the torsion-free. For $2.$, note that since $\nabla$ is compatible with $\left< \cdot, \cdot \right>$, we have that 
\begin{align*}
    \mathcal{L}_{\phi(\sigma)} \left< \phi(\xi), \phi(\eta)\right>(e) &= \left< \nabla_{\phi(\sigma)} \phi(\xi)(e), \phi(\eta)(e) \right> + \left< \phi(\xi)(e), \nabla_{\phi(\sigma)}\phi(\eta)(e) \right> \\
    &=\left<\nabla_{\sigma}^{\g} \xi, \eta \right> + \left< \xi, \nabla_{\sigma}^{\g}\eta \right>,
\end{align*} where $\mathcal{L}$ stands for the Lie derivative of vector fields.
But since $\phi(\xi), \phi(\eta)$ are left-invariant, $\left< \phi(\xi), \phi(\eta)\right>$ is a constant function, so  $\mathcal{L}_{\phi(\sigma)} \left< \phi(\xi), \phi(\eta)\right> \equiv 0$. \end{proof} %$\hfill\square$

\begin{remark}\label{remark: covg_operator}
We may consider the Riemannian $\g$-connection as an operator\newline $\nabla^\g: C^{\infty}([a, b], \g)\times C^{\infty}([a, b], \g) \to C^{\infty}([a, b], \g)$ in a natural way,  namely, if $\xi, \eta \in C^{\infty}([a, b], \g)$, we can write $(\nabla^\g_{\xi} \eta)(t) := \nabla^\g_{\xi(t)}\eta(t)$ for all $t \in [a, b]$. With this notation, Lemma \ref{lemma: covg_prop} works identically if we replace $\xi, \eta, \sigma \in \g$ with $\xi, \eta, \sigma \in C^{\infty}([a, b], \g)$.\hfill$\diamond$
\end{remark}

%\todo{ad and ad$^{*}$ were never defined}

We denote $ad^\dagger_{\xi} \eta = (\ad^\ast_{\xi} \eta^\flat)^\sharp$ for all $\xi, \eta \in \g$, which leads to the following decomposition of $\nabla^\g$ (see Theorem $5.40$ of \cite{bullo2019geometric}, for instance):
\begin{lemma}\label{lemma: covg-decomp}
The Riemannian $\g$-connection can be expressed as:
\begin{align*}
    \nabla_{\xi}^\g \eta = \frac12\left( [\xi, \eta]_\g - \ad^\dagger_{\xi} \eta - \ad^\dagger_{\eta} \xi\right)
    \end{align*} for all $\xi, \eta \in \g$.
\end{lemma}

Given a basis $\{A_i\}$ of $\g$, we may write any vector field $X$ on $G$ as $X = X^i \phi(A_i)$, where $X^i: G \to \R$, where we have adopted the Einstein sum convention. If $X$ is a vector field along some smooth curve $g: [a, b] \to G$, then we may equivalently write $X = X^i g A_i$, where now $X^i: [a, b] \to \R$ and $g A_i =: L_g A_i$. We denote $\dot{X} = \dot{X}^i g A_i$, which may be written in a coordinate-free fashion via $\dot{X}(t) = \frac{d}{dt}\left(L_{g(t)^{-1 \ast}}X(t) \right)$. We now wish to understand how the Levi-Civita connection $\nabla$ along a curve is related to the Riemannian $\g$-connection $\nabla^\g$. This relation is summarized in the following result:

\begin{lemma}\label{lemma: cov-to-covg}
Consider a Lie group $G$ with Lie algebra $\g$ and left-invariant Levi-Civita connection $\nabla$. Let $g: [a,b] \to G$ be a smooth curve and $X$ a smooth vector field along $g$. Then the following relation holds for all $t \in [a, b]$:
\begin{align}
    D_t X(t) = g(t)\left(\dot{X}(t) + \nabla_{\xi}^\g \eta(t) \right).\label{eqs: Cov-to-covg}
\end{align}
\end{lemma}
\begin{proof} Let $\{A_i\}$ be a basis for $\g$, and suppose that $\xi(t) = g(t)^{-1} \dot{g}(t) = \xi^i(t) A_i$ and $\eta(t) = g(t)^{-1} X(t) = \eta^i(t) A_i$. Expanding out the left-hand side of \eqref{eqs: Cov-to-covg}, and using that $L_{g^{*}}(\eta_jA_j)=\eta_jL_{g^{*}}(A_j)$ we get:
\begin{align*}
    D_t X(t) &= D_t\left(\eta^j g A_j \right)(t) \\
    &= g(t)\frac{\partial \eta^j}{\partial t}(t)A_i + \eta^j(t) D_t \left(gA_j\right).
\end{align*}
Note that $g(t)A_i = \phi(A_i)(g(t))$ for all $t \in [a, b], 1 \le i \le n$. Hence, 
\begin{align*}
    D_t \left(gA_j(t)\right) &= \xi^i(t)\left(\nabla_{\phi(A_i)}\phi(A_j)  \right)(g(t)), \\
    &= g(t) \xi^i(t)\left(\nabla_{\phi(A_i)} \phi(A_j)  \right)(e) \\
    &= g(t)\xi^i(t) \nabla_{A_i}^\g A_j,
\end{align*}
since $\phi(A_i)$ is left-invariant for all $1 \le i \le n$ and left-translation is an isometry. Equation \eqref{eqs: Cov-to-covg} then follows by the bi-linearity of $\nabla^\g$ together with the definition of $\dot{X}$ with respect to the basis $\{A_i\}$.
\end{proof}

We have seen in Section \ref{Sec: background} that the critical points $g: [a, b] \to G$ of the energy functional $\mathcal{E}$ are geodesics, that is $D_t \dot{g} = 0$ on the full interval $[a, b]$. Observe that from Lemma \ref{lemma: cov-to-covg}, this implies that $L_{g^\ast} \left( \dot{\xi} + \nabla^\g_{\dot{\xi}} \dot{\xi}\right) = 0$ for $\xi := g^{-1} \dot{g}$, and since left-translation is a diffeomorphism, we obtain the following reduction of geodesics:
\begin{theorem}\label{thm: EP_geo}
Suppose that $g: [a, b] \to G$ is a geodesic, and let $\xi := g^{-1} \dot{g}$. Then, $\xi$ satisfies
\begin{align}
    \dot{\xi} + \nabla^\g_{\dot{\xi}} \dot{\xi} = 0\label{EP: geo}
\end{align}
on $[a, b]$.
\end{theorem}

\begin{remark}\label{remark: lagrangian}
Equation \eqref{EP: geo} is often called the \textit{Euler-Poincar\'e equations} for geodesics. The Euler-Poincar\'e equations may be formulated more generally in the language of Lagrangian mechanics. That is, given a function $L: TG \to \R$ called the \textit{Lagrangian}, we seek to minimize the action $\mathcal{A}[g] = \displaystyle{\int_a^b L(g, \dot{g})dt}$ on $\Omega$. It is well-known that the critical points of $\mathcal{A}$ satisfy the \textit{Euler-Lagrange equations} $\displaystyle{\frac{d}{dt} \left(\frac{\partial L}{\partial \dot{g}}\right) + \frac{\partial L}{\partial g}} = 0$. If the Lagrangian is \textit{left-invariant}, meaning that $L(hg, h\dot{g}) = L(g, \dot{g})$ for all $g, h \in G$, then we may consider the \textit{reduced Lagrangian} $l: \g \to \R$ given by $l(\xi) := L(e, \xi) = L(g, \dot{g})$, where $\xi := g^{-1} \dot{g}$. Then, the critical points of the \textit{reduced action} $\mathcal{A}_{\text{red}}[\xi] = \displaystyle{\int_a^b l(\xi)dt}$ among variations of the from $\delta \xi = \dot{\eta} + [\xi, \eta]_\g$, where $\eta$ is arbitrary admissible proper variation, satisfy the Euler-Poincar\'e equations $\displaystyle{\frac{d}{dt}\left(\frac{\partial l}{\partial \xi}\right) = \ad^\ast_{\xi} \frac{\partial l}{\partial \xi}}.$ Moreover, it can be see that $g$ satisfies the Euler-Lagrange equations corresponding to $L$ if and only if $\xi := g^{-1} \dot{g}$ satisfies the Euler-Poincar\'e equations corresponding to $l$ \cite{holm1998euler}. The equation $\dot{g} = g \xi$ is known as the \textit{reconstruction equation}. The reduced variational principle is an example of a \textit{constrained} variational principle, since we are constraining the set of admissible variations. We will revisit constrained variational principles in Section \ref{Sec: Reduction_Homo} for a second-order variational problem on homogeneous spaces.

%\todo{(1) El principio variacional reducido es un ``constrained variational principle''como lo llaman en el articulo.
%(2) especifica cuales son las variaciones para el elemento $\xi$ que vienen dadas por $\eta$ tambien (i guess in your notation in sect 3 is $\sigma$ who plays the role of $\eta$ in the paper of Holm, Marsden an Ratiu.}

The energy function $\mathcal{E}$ is precisely the action corresponding to the Lagrangian $L(g, \dot{g}) = \frac12\|\dot{g}\|^2$. Since the Riemannian metric is left-invariant, it follows that the Lagrangian too is left-invariant, and so the reduced Lagrangian takes the form $l(\xi) =\frac12 \|\xi\|^2$. It is straight-forward to show that $\displaystyle{\frac{\partial l}{\partial \xi} = \xi^\flat}$, so that the Euler-Poincar\'e equations associated to $l$ are given by $\dot{\xi}^\flat = \ad_\xi^\ast \xi^\flat$, which is equivalent to \eqref{EP: geo} by Lemma \ref{lemma: covg-decomp}. Notice that the Euler-Poincar\'e equations corresponding to $l$ naturally live on the dual of the Lie algebra $\g^\ast$. It is only through the metric that we are able to convert them into equations on $\g$.
\end{remark}

%\begin{proof} Suppose that $\Gamma: (-\epsilon, \epsilon) \times [a, b] \to G$ is a one-parameter variation of $g$ with:
%\begin{align*}
%    T(s,t) &:= \partial_t \Gamma(s,t), &&\Tg(s,t) := \Gamma^{-1}(s,t) T(s,t), \\
%    S(s,t) &:= \partial_s \Gamma(s,t), &&\Sg(s,t) := \Gamma^{-1}(s,t) S(s,t). \\
%\end{align*}
%Then, \begin{align*}
%    \delta \mathcal{E} &= \int_a^b \left<\nabla_S T, T \right>\big\vert_{s=0}dt, \\
%    &= \int_a^b \left<\Tg' + \nabla_\Sg^\g \Tg, \Tg \right>\big\vert_{s=0}dt, \\
%    &= \int_a^b \left<\dot{\Sg} + \left[\Tg, \Sg \right], \Tg \right>\big\vert_{s=0}dt + \int_a^b \left<\nabla_\Sg^\g \Tg, \Tg \right>\big\vert_{s=0}dt, \\
%    &= -\int_a^b \left<\dot{\Tg}, \Sg \right>\big\vert_{s=0}dt + \int_a^b \frac{d}{dt}\left<\Sg, \Tg\right>\big\vert_{s=0}dt + \int_a^b \left<\nabla^\g_\Tg \Sg, \Tg\right>\big\vert_{s=0}dt, \\
%    &= -\int_a^b \left<\dot{\Tg} + \nabla^\g_\Tg \Tg, \Sg \right>\big\vert_{s=0}dt, \\
%    &= -\int_a^b \left<\dot{\xi} + \nabla_{\xi}^\g \xi, g^{-1} \delta g\right>dt.
%\end{align*}
%Since this must hold for all variational vector fields $\delta g$, and it is necessary that $\delta \mathcal{E} = 0$ along a critical point, equation \eqref{EP-eqs energy functional} follows immediately. \end{proof}

\subsection{Example 1: Geodesics for Rigid Body on $\SO(3)$}\label{section: example_geo_SO3}

%Lagrangians of the form $L(g, \dot{g}) = \frac{1}2\|\dot{g}\|^2$ are often called \textit{kinetic} due to their relation to kinetic energy in physics. A larger class, called \textit{mechanical} Lagrangians, are given by $L(g, \dot{g}) = \frac{1}2\|\dot{g}\|^2 + V(g)$, where $V: G \to \R$ is some smooth scalar field. If $V$ is left-invariant, the corresponding 

It is well-known that the attitude of a rigid body can be modelled on $G = \SO(3)$ equipped with the left-invariant Riemannian metric $\left<\dot{R}_1, \dot{R}_2\right> = \tr(\dot{R}_1 \mathbb{M} \dot{R}_2^T)$ for all $R \in \SO(3), \dot{R}_1, \dot{R}_2 \in T_R \SO(3)$, where $\mathbb{M}$ is a symmetric positive-definite $3\times 3$ matrix called the \textit{coefficient of inertia matrix}. On the Lie algebra $\so(3)$, the metric takes the form of the inner-product $\left<\hat{\Omega}_1, \hat{\Omega}_2\right> = \Omega_1^T \mathbb{J} \Omega_2$, where $\hat{\cdot}: \R^3 \to \so(3)$ is the \textit{hat isomorphism} defined by $(x_1, x_2, x_3)^T \mapsto \begin{bmatrix} 0 & -x_3 & x_2 \\ x_3 & 0 & -x_1 \\ -x_2 & x_1 & 0 \end{bmatrix}$, and $\mathbb{J}$ is a symmetric positive-definite $3\times 3$ matrix called the \textit{moment of inertia tensor}. 

Suppose that $G$ is equipped with the Levi-Civita connection induced by the above metric, and suppose that $R: [a, b] \to \SO(3)$ is a geodesic. Define $\hat{\Omega} := R^{-1} \dot{R} \in \so(3)$. Then, from Theorem \ref{thm: EP_geo}, we have $\dot{\hat{\Omega}} = \ad^\dagger_{\hat{\Omega}}\hat{\Omega}$, or $\dot{\hat{\Omega}}^\flat = \ad_{\hat{\Omega}}^\ast \hat{\Omega}^\flat$. First observe that $\hat{\Omega}^\flat(\hat{\eta}) := \left<\hat{\Omega}, \hat{\eta}\right> = \Omega^T \mathbb{J} \eta = (\mathbb{J}\Omega)^T \eta$ for all $\hat{\eta} \in \so(3)$. Hence, we may identify $\hat{\Omega}^\flat$ with $\mathbb{J}\Omega$ under the hat isomorphism. Moreover, it is well known that $\ad_{\hat{\Omega}}^\ast \bar{\Pi} = \Pi \times \Omega$ for all $\bar{\Pi} \in \so(3)^\ast$, where $\Pi$ is determined by the relationship $\left<\bar{\Pi}, \hat{\Omega}\right> = \Pi^T \Omega$. It then follows that \eqref{EP: geo} is equivalent to $\mathbb{J}\dot{\Omega} = \mathbb{J}\Omega \times \Omega$, which is recognized as Euler's equation for a rigid body.

%\todo{(1) Incluir como aplicacion o ejemplo o remark, la extension de las ecuaciones de euler-Poincar\'e para producto semidirecto y consecuentemente con (3) aplicacion a control optimo con symmetry breaking. (2) Conecta con el heavy top para cubrir los ejemplos clasicos de mecanica geometrica como hiciste con el rigid body.}

\section{Necessary and Sufficient Conditions for the Variational Obstacle Avoidance Problem}\label{Sec: Necessary conditions}
%\subsection{Problem formulation I}

%\todo{connect the problem with control systems. That is, you have a control system  - for instance double integrator dynamics - and you wish to minimize the norm of the acceleration and avoid obstacles. See for instance \url{https://arxiv.org/pdf/1910.09514.pdf} and \url{https://arxiv.org/abs/2003.12183}}

Consider a complete and connected Riemannian manifold $Q$. For some $X, Y \in TQ$ and $a < b \in \R$, let $\Omega_{X, Y}^{a, b}$ be defined as in the previous section. We define the function $J: \Omega \to \R$ by
\begin{equation}\label{J}
J(q)=\int\limits_0^T \Big{(}\frac12\Big{|}\Big{|}D_t\dot{q}(t)\Big{|}\Big{|}^2 + V(q(t))\Big{)}dt.
\end{equation}

%The functional is constructed as...
\textbf{(P1): Variational obstacle avoidance problem:} Find a curve $q\in \Omega$ minimizing the functional $J$, where $V:Q \to\mathbb{R}$ is a smooth and non-negative function called the \textit{artificial potential}, and ${D}_t \dot{q}$ is the covariant derivative along $q$ induced by the Levi-Civita connection on $Q$.

\vspace{.2cm}

In order to minimize the functional $J$ among the set $\Omega$, we want to find curves $q\in\Omega$ such that $J(q)\leq J(\tilde{q})$ for all
admissible curves $\tilde{q}$ in a $C^1$-neighborhood of $q$. Necessary conditions can be derived by finding $q$ such that the differential of $J$ at $q$, $dJ(q)$, vanishes identically. This is clearly equivalent to $dJ(q)W = 0$ for all $W \in T_q \Omega$, which itself can be understood through variations—as discussed in the previous section. The next result from \cite{BlCaCoCDC} characterizes necessary conditions for optimality in the variational obstacle avoidance problem.

%Imposing that $\frac{d}{dr}J(x_1,...,\Gamma_i(r),...x_n)\Big{|}_{r=0}=0$, we obtain that:

\begin{proposition}\label{th1}\cite{BlCaCoCDC} $q \in \Omega$ is a critical point for the functional $J$ if and only if it is a $\mathcal{C}^\infty$-curve on $[a,b]$ satisfying
\begin{equation}\label{eqq1}
    D^3_t \dot{q}+R\big{(}D_t\dot{q},\dot{q}\big{)}\dot{q} + \hbox{\grad} \, V(q(t)) = 0.
\end{equation}
\end{proposition}

\begin{remark}\label{remark_cubics}
We call the solutions to \eqref{eqq1} \textit{modified Riemannian cubic polynomials with respect to $V$}, or simply modified cubics when the potential is understood from context, following the standard nomenclature that such curves are referred to as Riemannian cubic polynomials in the event that $V \equiv 0$. Given a point obstacle $q_0 \in Q$, we may consider a potential of the form $V(q) = \frac{\tau}{1 + (d(q_0, q)/D)^{2N}}$, where $\tau, D \in \R^+$ and $N \in \N$. From \cite{goodman2022collision} it follows that the solutions to \eqref{eqq1} avoid the obstacle $q_0$ within a tolerance of $D$ (that is, they do not enter the ball $B_D(q_0) \subset Q$) for $\tau, N$ sufficiently large. Of course, this is not the only family of artificial potentials with such a property, however it should be noted that in general, such families will involve the Riemannian distance function, as it is the primary tool for making comparisons between the agent and the obstacle.
\end{remark}

\subsection{Reduction in the Variational Obstacle Avoidance Problem on Lie Groups}\label{reduction_obs_avoid_G}

We wish to obtain Euler-Poincar\'e equations corresponding to \eqref{eqq1} in the case that $Q = G$, under the additional assumptions:
\begin{quote}
\begin{description}
    \item[\textbf{G1:}] $G$ a Lie group endowed with a left-invariant Riemannian metric and corresponding Levi-Civita connection $\nabla$.
    \item[\textbf{G2:}] There exists a smooth function $V_{\ext}: G \times G \to \R$ called the \textit{extended artificial potential} which satisfies $V_{\ext}(\cdot, g_0) = V$ for some $g_0\in G$ and which is invariant under left-translation on $G\times G$. That is, $V_{\ext}(hg, h\bar{g}) = V_{\ext}(g, \bar{g})$ for all $g, \bar{g}, h \in G$.
\end{description}
\end{quote}

\noindent To do this, we first must understand the forms that $D_t^3 \dot{g}$ and $R(D_t \dot{g}, \dot{g})\dot{g}$ take when left-translated to curves in the Lie algebra. This is summarized in the following Proposition:

\begin{proposition}\label{prop: cubic_red}
Let $g: [a, b] \to G$ be a smooth curve and set $\xi := g^{-1}\dot{g}$ and $\eta = \dot{\xi} - \ad^\dagger_\xi \xi$. Then,
\begin{align}
    D^3_t \dot{g} &= g\Big{(}\ddot{\eta} + 2\nabla^\g_\xi \dot{\eta} + \nabla^\g_{\eta} \eta + \nabla^\g_{\ad^\dagger_\xi \xi} \eta + \nabla^\g_\xi \nabla^\g_\xi \eta\Big{)},\label{covTTTT} \\
    R\left(D_t \dot{g}, \dot{g}\right)\dot{g} &= g R\big{(}\eta, \xi \big{)}\xi.\label{R_g}
\end{align}
\end{proposition}
\begin{proof}
Observe that, from Lemma \ref{lemma: cov-to-covg}, we have $D_t^2 \dot{g} = D_t g \eta = g \left( \dot{\eta} + \nabla_{\xi}^\g \eta\right)$, so that \eqref{covTTTT} is calculated as:

\begin{align*}
    D_t^3 \dot{g} &= D_t g \left( \dot{\eta} + \nabla_{\xi}^\g \eta\right) \\
    &= \ddot{\eta} + 2\nabla^\g_\xi \dot{\eta} + \nabla^\g_{\dot{\xi}} \eta + \nabla^\g_\xi \nabla^\g_\xi \eta \\
    &= \ddot{\eta} + 2\nabla^\g_\xi \dot{\eta} + \nabla^\g_{\eta} \eta + \nabla^\g_{\ad^\dagger_\xi \xi} \eta + \nabla^\g_\xi \nabla^\g_\xi \eta
\end{align*}
For \eqref{R_g}, we see from Lemma \ref{lemma: cov-to-covg} that $D_t \dot{g} = g \eta$, and hence:
\begin{align*}
    R\left(D_t \dot{g}, \dot{g}\right)\dot{g} &= R\big{(}g\eta, g\xi \big{)}g\xi \\
    &= g R\big{(}\eta, \xi \big{)}\xi
\end{align*}
since $R$ is invariant under left-translation by $g(t)$ for any fixed $t \in [a, b]$ and since $R$ is a tensor field (and hence only depends on its arguments at the point at which we are evaluating it).
\end{proof}
%\nabla_S \nabla_T T &= \Gamma\left( (\dot{\Tg})' + \nabla^\g_{\Sg} \dot{\Tg} + \nabla^\g_{\Tg'} \Tg + \nabla^\g_{\Tg} \Tg' + \nabla^\g_{\Sg} \nabla_{\Tg}^\g \Tg \right) \\
    %\nabla_T \nabla_T T &= \Gamma\left( \ddot{\Tg} + 2 \nabla^\g_{\Tg} \dot{\Tg} + \nabla^\g_{\dot{\Tg}} \Tg + \nabla^\g_{\Tg} \nabla_{\Tg}^\g \Tg \right) \\
    %\nabla_T \nabla_S T &= \Gamma\left( (\dot{\Tg})' + \nabla^\g_{\Tg} \Tg' + \nabla^\g_{\dot{\Sg}} \Tg + \nabla^\g_{\Sg} \dot{\Tg} + \nabla^\g_{\Tg} \nabla_{\Sg}^\g \Tg \right) \\

The quantities calculated in Proposition \ref{prop: cubic_red} may be substituted directly into equation \eqref{eqq1}. If $g$ were a cubic polynomial (i.e., in the case that $V \equiv 0$), we would immediately obtain reduced equations on the Lie algebra $\g$. However, in the case that the artificial potential is non-trivial, this is only possible if $V$ is also left-invariant—which is not generally true. However, it is often the case that $V$ is left-invariant on some stabilizer subgroup of $G$. Indeed, assumption \textbf{G2} is equivalent to the assumption that $V$ is left-invariant on the \textit{stabilizer subgroup} of $g_0$, $\text{Stab}(g_0)$ (also called the \textit{isotropy subgroup} of $g_0$). For the purposes of obstacle avoidance, $g_0$ takes the form of a point-obstacle that we wish to avoid. As for classical Euler-Poincar\'e reduction, we will be able to use these assumption to reduce \eqref{eqq1} to a set of equations on $\g$, together with some reconstruction equations, by working directly with the variational principle. This is seen in the following proposition.

%\todo{Define $\text{Stab}(g_0)$ and $\grad_1$}

\begin{proposition}\label{prop: reduction_left_inv}
Consider the variational obstacle avoidance problem \textbf{P1} with $Q = G$ and the additional assumptions \textbf{G1, G2}. Then $g \in \Omega$ solves \eqref{eqq1} if and only if $\xi := g^{-1} \dot{g}$ and $h := g_0^{-1}g$ solve:
\begin{align}
    \dot{\xi} &= \eta + \ad^\dagger_{\xi} \xi, \label{ad_dag}\\
   0&= \ddot{\eta} + 2\nabla^\g_\xi \dot{\eta} + \nabla^\g_{\eta} \eta + \nabla^\g_{\ad^\dagger_\xi \xi} \eta + \nabla^\g_\xi \nabla^\g_\xi \eta + R\big{(}\eta, \xi \big{)}\xi + L_{h^{-1 \ast}} \grad_1 V_{\ext}(h, e), \label{eqq2} \\
    \dot{h} &= L_{h^\ast} \xi. \label{eqqalpha}
\end{align}
\end{proposition}
\begin{proof}
From Proposition \ref{th1} and assumption \textbf{G2}, it is clear that
\begin{align}
    \delta J(g) &= \int_a^b \left< D_t^3 \dot{g} + R\big{(}D_t \dot{g}, \dot{g}\big{)}\dot{g}, \delta g \right>dt + \delta\int_a^b V_{\ext}(h, e)dt\label{variation}
\end{align}
Now let $\sigma(t) := g(t)^{-1}\delta g(t)$. Then, $\delta h = g_0^{-1} \delta g = g_0^{-1} g(t) \sigma(t) = h(t) \sigma(t) = L_{h(t)^\ast} \sigma(t)$ and:
\begin{align*}
    \delta\int_a^b V_{\ext}(h, e)dt &= \int_a^b \frac{\partial}{\partial s}\Big{\vert}_{s=0} V_{\ext}(h, e)dt \\
    &= \int_a^b \left<\frac{\partial V_{\ext}}{\partial h}, L_{h(t)\ast}\sigma\right>_{\g^\ast}dt \\
    &= \int_a^b \left<\grad_1 V(h, e), L_{h(t)\ast}\sigma\right> dt \\
    &= \int_a^b \left<L_{h(t)^{-1 \ast}}\grad_1 V(h. e), \ \sigma\right> dt
\end{align*}
Plugging this into \eqref{variation} and using Proposition \ref{prop: cubic_red}, we obtain
\begin{equation*}
    \int_a^b \left<\ddot{\eta} + 2\nabla^\g_\xi \dot{\eta} + \nabla^\g_{\eta} \eta + \nabla^\g_{\ad^\dagger_\xi \xi} \eta + \nabla^\g_\xi \nabla^\g_\xi \eta + R\big{(}\eta, \xi \big{)}\xi + L_{h(t)^{-1 \ast}}\grad_1 V(h, e), \ \sigma \right>dt = 0.
\end{equation*}
The conclusion follows upon applying the fundamental lemma of the calculus of variations, together with the observation that $\dot{h} = g_0^{-1} \dot{g}= g_0^{-1}gg^{-1}\dot{g}= h \xi$.
\end{proof}

\begin{remark}
As in the first order case discussed in Remark \ref{remark: lagrangian}, problem \textbf{P1} is a special case of a higher order theory of Lagrangians. This was studied on Lie groups in \cite{gay2012invariant}, where the corresponding higher order Euler-Poincar\'e equations were obtained. However, this formalism requires the use of higher order tangent bundles—whereas the formalism adopted here allows to work only on the tangent bundle $TG$.  Moreover, through Lemma \ref{lemma: cov-to-covg}, we are able to directly reduce the necessary conditions corresponding to some variational principle given the appropriate symmetries rather than reduce the variational principle itself—which is not possible in the Lagrangian formalism. This can be seen in the proof of Proposition \ref{prop: reduction_left_inv} with the reduction of the component $D_t^3 \dot{g} + R(D_t \dot{g}, \dot{g})\dot{g}$. Observe also that we have $L_{h(t)^{-1 \ast}} \grad_1 V(h, e) = J_L (\grad_1 V_{\ext}(h, e))^\sharp$, where $J_L: T^\ast G \to \g^\ast$ is the \textit{momentum map} corresponding to the left-action of $G$ on $T^\ast G$ (that is, the cotangent lift of the left-translation map). 
\end{remark}

\subsection{Reduction on Lie Groups with Bi-invariant Metrics}\label{sec: red_biinv}

We call a Riemannian metric $\left< \cdot, \cdot \right>{\Bi}$ on $G$ \textit{bi-invariant} if it is invariant under both left- and right-translations, or equivalently, is $\Ad$-invariant, that is, invariant under the adjoint action of $G$ on $\mathfrak{g}$. Unlike left-invariant and right-invariant metrics, not every Lie group admits a bi-invariant metric. However, it is known that any compact and connected Lie group does (although this is not a necessary condition). Hence, we view equipping $G$ with such a bi-invariant metric as a strengthening of assumption \textbf{G1}.
\begin{quote}
\noindent$\textbf{G1}^\ast$\textbf{:} $G$ is a Lie group equipped with a bi-invariant Riemannian metric and corresponding Levi-Civita connection $\nabla$.
\end{quote}

\noindent Under assumption $G1^\ast$, equation \eqref{eqq2} simplifies considerably. This is due to the following lemma from \cite{Milnor} (applied in the context of the Riemannian $\g$-connection):
\begin{lemma}\label{bi-invariant_G}
If $G$ is a Lie group equipped with a bi-invariant Riemannian metric and Levi-Civita connection $\nabla$, and $\xi, \eta, \sigma \in \g$, then:
\begin{align}
    \nabla_{\xi}^\g \eta &= \frac12 \left[\xi, \eta\right]_{\g},\label{cov_bi} \\
    R\big{(}\xi, \eta{)}\sigma &= -\frac14 \big{[}\big{[}\xi, \eta \big{]}_{\g}, \sigma \big{]}_\g.\label{R_bi}
\end{align}
\end{lemma}

\noindent We immediately obtain the following corollary to Proposition \ref{prop: reduction_left_inv}.
\begin{corollary}\label{cor: reduction_bi_inv}
Consider the variational obstacle avoidance problem \textbf{P1} with $Q = G$ and the additional assumptions $\textbf{G1}^\ast$, \textbf{G2}. Then $g \in \Omega$ solves \eqref{eqq1} if and only if $\xi := g^{-1} \dot{g}$ and $h := g_0^{-1}g$ solve:
\begin{align}
    \dddot{\xi} &+ \big{[}\xi, \ddot{\xi}\big{]}_\g + L_{h^{-1 \ast}} \grad_1 V_{\ext}(h, e)= 0, \label{eqq3} \\
    \dot{h} &= L_{h^\ast}\xi. \label{eqqalpha_bi}
\end{align}
\end{corollary}
\begin{proof}
Observe that, as a consequence of equation \eqref{cov_bi}, the following two relations hold for all $\xi, \eta \in \g$:
\begin{align*}
    \nabla^\g_{\xi} \eta &= -\nabla^\g_{\eta} \xi, \\
    \nabla^\g_{\xi} \xi &= 0.
\end{align*}
Moreover, the curvature endomorphism relates to the Riemannian $g$-connection as follows:
\begin{align*}
    R(\xi, \eta)\sigma &= -\frac14 \big{[}\big{[}\xi, \eta \big{]}_{\g}, \sigma \big{]}_\g, \\
    &= -\frac14 \big{[}\sigma, \big{[}\eta, \xi \big{]}_{\g} \big{]}_\g, \\
    &= -\frac12 \big{[}\sigma, \nabla^\g_{\eta} \xi \big{]}_\g \\
    &= -\nabla^\g_{\sigma} \nabla^\g_{\eta} \xi.
\end{align*}
Applying these identities to \eqref{eqq2} term-by-term yields \eqref{eqq3}.
\end{proof}

Following Remark \ref{remark_cubics}, consider a potential of the form $V(g) = \frac{\tau}{1 + (d(g, g_0)/D)^{2N}}$. We may consider an extended potential $V_{\ext}: G\times G \to \R$ given by $V_{\ext}(g_1, g_2) = \frac{\tau}{1 + (d(g_1, g_2)/D)^{2N}}$. In order for assumption \textbf{G2} to apply, we need that the Riemannian distance $d$ is invariant under left-translation, which we show in the following lemma.

\begin{lemma}\label{lemma: distance-left-inv}
$d(gq, gp) = d(q, p)$ for all $g,q, p \in G$.
\end{lemma}
\begin{proof}
Since $G$ is complete as a Riemannian manifold, there exists a geodesic $\gamma: [0, 1] \to G$ which minimizes the length functional $\displaystyle{L(c) = \int_0^1 \|\dot{c}(t)\| dt}$ among all smooth curves $c:[0, 1] \to G$ satisfying $\ c(0) = p, \ c(1) = q$. Moreover, we have $d(p, q) = L(\gamma)$. By left-invariance of the metric, we then have that $d(p, q) = L(\gamma) = L(g\gamma) \ge d(gp, gq)$, since in particular $g \gamma$ is a smooth curve such that $g\gamma(0) = gp, \ g\gamma(1) = gq$. On the other hand, there exists some geodesic $\gamma^\ast$ such that $L(\gamma^\ast) = d(gp, gq)$, and so $d(gp, gq) = L(\gamma^\ast) = L(g^{-1}\gamma^\ast) \ge d(p, q)$. It follows that $d(p, q) = d(gp, gq)$.
\end{proof}

Hence $V_{\ext}$ is left-invariant under left translation on $G \times G$. Evaluating $V_{\ext}(h, e)$ in practice—let alone its gradient vector field—is highly non-trivial. This is due to the fact that a separate calculation for $d(h(t), e)$ is needed for each $t \in [a, b]$ unless a closed form for all minimizing geodesics can be found, which is generally not the case. If we assume that for all $t \in [a, b]$, $h(t)$ and $e$ are simultaneously contained within some geodesically convex neighborhood on which the Riemannian exponential map is a diffeomorphism, then we have $d(h, e) = \|\exp_h^{-1}(e)\|$. In such a case, consider the family $\gamma: [0, 1] \times [a, b] \to G$ of geodesics given by $\gamma(s, t) = \exp_{h(t)}(s \exp^{-1}_{h(t)}(e))$, which satisfies $\gamma(0, t) = h(t)$ and $\gamma(1, t) = e$ for all $t \in [a, b]$. It then follows that
\begin{align*}
    \frac{d}{dt}d^2(h(t), e) &= \frac{d}{dt}\left(\int_0^1 \|\partial_s \gamma\|ds \right)^2 \\
    &= 2\int_0^1 \left<D_t \partial_s \gamma, \partial_s \gamma \right>ds \\
    &= 2\int_0^1 \frac{d}{ds}\left<\partial_t \gamma, \partial_s \gamma \right>ds \\
    &= \left<\dot{h}(t), -2\exp_{h(t)}^{-1}(e) \right>,
\end{align*}
from which we see that $\grad_1 d^2(h(t), e) =-2\exp_{h(t)}^{-1}(e)$. By a standard application of the chain rule, and the fact that $h(t)^{-1} \exp^{-1}_{h(t)}(e) = -\exp^{-1}_{e}(h(t))$, we obtain:
\begin{equation}\label{potential_exp}
    L_{h(t)^{-1 \ast}}\grad_1 V_{\ext}(h(t), e) = -\frac{2N\tau \|\exp^{-1}_{e}(h(t))\|^{2N-2}}{D^{2N}(1 + (\|\exp^{-1}_{e}(h(t))\|/D)^{2N})^2}\exp^{-1}_{e}(h(t)).
\end{equation} 
When the metric is only left-invariant, equation \eqref{potential_exp} typically suffers from the same problems that calculating $d(h(t), e)$ does. However, under assumption $\textbf{G1}^\ast$, $\exp_e$ agrees with the \textit{Lie exponential map} $\Exp: \g \to G$ defined for all $\xi \in \g$ by $\Exp(\xi) = \gamma(1)$, where $\gamma$ is the unique solution to $\dot{\gamma} = \gamma \xi$ with $\gamma(0) = e$. That is, the one-parameter subgroups of $G$ are exactly the geodesics through the identity. Wherever defined, the inverse of Lie exponential map $\Log: G \to \g$ is called the \textit{Logarithmic map}. Hence, equation \eqref{potential_exp} is equivalent to 
\begin{equation}\label{potential_Log}
    L_{h(t)^{-1 \ast}}\grad_1 V_{\ext}(h(t), e) = -\frac{2N\tau \|\Log(h(t))\|^{2N-2}}{D^{2N}(1 + (\|\Log(h(t))\|/D)^{2N})^2}\Log(h(t)).
\end{equation} 
If $G$ is a matrix Lie group, it can be seen that $\displaystyle{\Exp(A) = \sum_{k = 0}^\infty \frac{A^k}{k!}}$ for all $A \in G$. For certain subgroups of $GL(n)$, closed forms for $\Exp$ and $\Log$ can be obtained, which we will utilize for the special case of $G = \SO(3)$ in Section \ref{section: rigid_body_cubics}. 

From Section \ref{section: example_geo_SO3}, we see that $\left< \cdot, \cdot\right>_{\Bi}$ aligns with the left-invariant metric for a rigid body in the case that the coefficient of inertia matrix $\mathbb{M}$ (or equivalently that the moment of inertia tensor $\mathbb{J}$) is a scalar multiple of the identity. This is akin to saying that the rigid body is symmetric about all of its axes, which is a strong assumption that typically does not hold in application. Hence, even in the case that $G$ admits a bi-invariant metric, this metric may not correspond to the kinetic energy of the physical system that we are interested in studying, and a left-invariant metric must be used anyways. However, as we will see, we may still take advantage of the bi-invariant metric to design our artificial potential.

\begin{quote}
\noindent$\textbf{G1}^{\ast \ast}$\textbf{:} $G$ is a Lie group equipped with both a left-invariant and a bi-invariant Riemannian metric, denoted by $\left< \cdot, \cdot\right>$ and $\left< \cdot, \cdot\right>_{\Bi}$, respectively. Denote the Levi-Civita connection corresponding to $\left< \cdot, \cdot \right>$ by $\nabla$, and the corresponding Riemannian $\g$-connection and Riemannian curvature by $\nabla^\g$ and $R$, respectively. Let $\beta: \g \to \g$ be the linear endomorphism such that $\left<\xi, \eta\right>_{\Bi} = \left< \beta(\xi), \eta\right>$ for all $\xi, \eta \in \g$.
\end{quote}

\begin{proposition}\label{prop: reduction_left_inv_comp}
Consider the variational obstacle avoidance problem \textbf{P1} with $Q = G$ and the additional assumptions $\textbf{G1}^{\ast\ast}, \textbf{G2}$. Then $g \in \Omega$ solves \eqref{eqq1} if and only if $\xi := g^{-1} \dot{g}$ and $h := g_0^{-1}g$ solve:
\begin{align}
    \dot{\xi} &= \eta + \ad^\dagger_{\xi} \xi, \label{ad_dag_leftbi}\\
   0&= \ddot{\eta} + 2\nabla^\g_\xi \dot{\eta} + \nabla^\g_{\eta} \eta + \nabla^\g_{\ad^\dagger_\xi \xi} \eta + \nabla^\g_\xi \nabla^\g_\xi \eta + R\big{(}\eta, \xi \big{)}\xi + \beta(L_{h^{-1 \ast}} \grad_1^{\Bi} V_{\ext}(h, e)) , \label{eqq2_leftbi} \\
    \dot{h} &= L_{h^\ast}\xi, \label{eqh}
\end{align}
where $\grad_1^{\Bi} V_{\ext}(h. e)$ denotes the gradient vector field of $V_{\ext}$ with respect to its first argument and $\left<\cdot, \cdot\right>_{\Bi}$.
\end{proposition}
\begin{proof}
As in Proposition \ref{prop: reduction_left_inv_comp}, we let $\sigma(t) := g(t)^{-1}\delta g(t)$. Then, $$\delta h = g_0^{-1} \delta g = g_0^{-1} g(t) \sigma(t) = h(t) \sigma(t) = L_{h(t)^\ast} \sigma(t).$$ Hence,
\begin{align*}
    \delta\int_a^b V_{\ext}(h, e)dt &= \int_a^b \left<\frac{\partial V_{\ext}}{\partial h}, L_{h(t)\ast}\sigma\right>_{\g^\ast}dt \\
    &= \int_a^b \left<\grad_1^{\Bi} V(h, e), L_{h(t)\ast}\sigma\right>_{\Bi} dt \\
    &= \int_a^b \left<L_{h(t)^{-1 \ast}}\grad_1^{\Bi} V(h. e), \ \sigma\right>_{\Bi} dt \\
    &= \int_a^b \left<\beta(L_{h(t)^{-1 \ast}}\grad_1^{\Bi} V(h. e)), \ \sigma\right> dt,
\end{align*}
The remainder of the proof follows identically to that of Proposition \ref{prop: reduction_left_inv_comp}.
\end{proof}

We now choose our obstacle avoidance potential as $V(g) = \frac{\tau}{1 + (d_{\Bi}(g, g_0)/D)^{2N}}$, where $d_{\Bi}$ is the Riemannian distance function corresponding to the bi-invariant metric $\left<\cdot, \cdot\right>_{\Bi}$. The extended potential is similarly defined by $V_{\ext}(g_1, g_2) = \frac{\tau}{1 + (d_{\Bi}(g_1, g_2)/D)^{2N}}$. Assume that $h$ is contained within a geodesically convex neighborhood of $e$ on the full interval $[a, b]$, so that $d_{\Bi}(e, h) = \|\exp_e^{-1}(h)\|_{\Bi}$, where $\exp$ is the Riemannian exponential map corresponding to the Levi-Civita connection induced by $\left< \cdot, \cdot \right>_{\Bi}$. Since the metric is bi-invariant, we again have $\exp_e^{-1}(h) = \Log(h)$, and because we are taking the gradient of $V_{\ext}$ with respect to this metric, we obtain

\begin{equation}\label{potential_Log_left}
    \beta(L_{h(t)^{-1 \ast}}\grad_1^{\Bi} V_{\ext}(h(t), e)) = -\frac{2N\tau \|\Log(h(t))\|^{2N-2}}{D^{2N}(1 + (\|\Log(h(t))\|/D)^{2N})^2}\beta(\Log(h(t))).
\end{equation}

\subsection{Example 2: Necessary conditions for Rigid Body on $\SO(3)$}\label{section: rigid_body_cubics}
We return to the example of the rigid body modelled on $\SO(3)$. As in Section \ref{section: example_geo_SO3}, we may equip $\SO(3)$ with the left-invariant metric $\left<\dot{R}_1, \dot{R}_2\right> = \tr(\dot{R}_1 \mathbb{M} \dot{R}_2^T)$ for all $R \in \SO(3), \dot{R}_1, \dot{R}_2 \in T_R \SO(3)$. Identifying $\so(3)$ with $\R^3$ under the hat isomorphism, we have $ad^\dagger_{\xi} \sigma = \mathbb{J}^{-1}\left(\mathbb{J}\sigma \times \xi\right)$ for all $\xi, \sigma \in \R^3$ and so by Lemma \ref{lemma: covg-decomp}, the Riemannian $\g$-connection (with respect to the Levi-Civita connection of the left-invariant metric) takes the form 
\begin{equation}\label{gconSO3}
    \nabla^\g_{\xi} \sigma = \xi \times \sigma + \mathbb{J}^{-1}\left(\xi \times \mathbb{J}\sigma + \sigma \times \mathbb{J}\xi\right)
\end{equation}
for all $\xi, \sigma \in \R^3$.

We also consider the bi-invariant metric $\left< \cdot, \cdot \right>_{\Bi}$ defined by $\left< \dot{R}_1, \dot{R}_2\right>_{\Bi} = \tr(\dot{R}_1 \dot{R}_2^T)$ for all $R \in \SO(3), \dot{R}_1, \dot{R}_2 \in T_R \SO(3)$. Through the hat isomorphism, we then get $\left<\hat{\Omega}_1, \hat{\Omega}_2\right>_{\Bi} = \Omega_1^T \Omega_2$, which is just the standard inner product on $\R^3$. From this it is clear that $\beta(\hat{\Omega}) = \hat{\Omega}\mathbb{M}^{-1}$ for all $\hat{\Omega} \in \so(3)$, since $\left< \beta(\hat{\Omega}_1), \hat{\Omega}_2\right> = \tr((\hat{\Omega}_1 \mathbb{M}^{-1})\mathbb{M}\hat{\Omega}_2^T) = \tr(\hat{\Omega}_1\hat{\Omega}_2^T) = \left< \beta(\hat{\Omega}_1), \hat{\Omega}_2\right>_{\Bi}$ for all $\hat{\Omega}_1, \hat{\Omega}_2 \in \so(3)$. We may also consider $\beta$ to be a map from $\R^3$ to $\R^3$ as $\beta(\Omega) = \mathbb{J}^{-1}\Omega$ for all $\Omega \in \R^3$.

Consider a point-obstacle $R_0 \in \SO(3)$ and the artificial potential $V(R) = \frac{\tau}{1 + (d_{\Bi}(R, R_0)/D)^{2N}}$ with the extended potential $V_{\ext}(R_1, R_2) = \frac{\tau}{1 + (d_{\Bi}(R_1, R_2)/D)^{2N}}$, as in Section \ref{sec: red_biinv}. With the aim of using Proposition \ref{prop: reduction_left_inv_comp}, we now seek to calculate the logarithmic map in $\SO(3)$. This is provided by Proposition $5.7$ of \cite{bullo2019geometric}:
\begin{lemma}\label{Log}
$\Exp: \so(3) \to \SO(3)$ is diffeomorphism between $\{\hat{\Omega}: \Omega \in \R^3, \ \Omega^T \Omega \le \pi^2\}$ and $\{R \in \SO(3): \ \tr(R) \ne -1\}$, and the logarithmic map is given by
$$\Log(R) = \begin{cases} 0, &R = I \\ \frac{\phi(R)}{\sin(\phi(R))}(R - R^T), &R \ne I
\end{cases}$$
where $\phi(R) := \arccos(\frac12(\tr(R) - 1))$. Moreover, we have $\|\Log(R)\|_{\Bi} = \phi(R)$.
\end{lemma}

Using Lemma \ref{Log}, equation \ref{potential_Log_left} takes the form:
\begin{equation}\label{potential_Log_SO3}
    \beta(L_{H^{-1 \ast}}\grad_1 V_{\ext}(H, e)) = -\frac{2N\tau \phi(H)^{2N-1}}{\sin(\phi(H))D^{2N}(1 + (\phi(H)/D)^{2N})^2} (H - H^T)\mathbb{M}^{-1}.
\end{equation} 
for all $H \in \SO(3)$ with $\tr(H) \ne -1$. Hence, by Proposition \ref{prop: reduction_left_inv_comp}, $R$ solves \eqref{eqq1} if and only if $\hat{\Omega} := R^{-1} \dot{R}$ and $H := R_0^{-1} R$ solve:
\small{\begin{align}
    \mathbb{J}\dot{\Omega} &= \mathbb{J}\Omega \times \Omega + \mathbb{J}\eta \label{ad_dag_SO3}\\
    \frac{2N\tau \phi(H)^{2N-1}}{\sin(\phi(H))D^{2N}(1 + (\phi(H)/D)^{2N})^2} \mathbb{J}^{-1}(H - H^T)^\vee \label{eqq2_SO3} &=\ddot{\eta} + 2\nabla^\g_\Omega \dot{\eta}\\ &+ \nabla^\g_{\eta} \eta + \nabla^\g_{\mathbb{J}\Omega \times \Omega} \eta + \nabla^\g_\Omega \nabla^\g_\Omega \eta + R\big{(}\eta, \Omega \big{)}\Omega \nonumber \\
    \dot{H} &= H\Omega, \label{eqqalpha_SO3}
\end{align}}where $\nabla^\g$ is calculated as in equation  \eqref{gconSO3}, and the Riemannian curvature can be found by $R(\eta, \xi)\xi = \nabla^\g_\eta \nabla^\g_\xi \xi - \nabla^\g_\eta \nabla^\g_\xi \xi - \nabla^\g_{\eta \times \xi} \xi$.

\section{Reduction on Riemannian Homogeneous Spaces with Broken Symmetry}\label{Sec: Reduction_Homo}
%Suppose that $G$ is a Lie group endowed with a left-invariant Riemannian metric and $K$ is a closed Lie subgroup of $G$. Then it is well-known that the quotient space $H := G/K$ is a Riemannian symmetric space. Moreover, the relations $[\s, \s] \subset \s, [\m, \m] \subset \s, [\m, \s] \subset \m$ hold. Using this decomposition of $T_gG$, it is possible to extend the notion of vertical and horizontal tangent vectors on G to vertical and horizontal vector fields and curves.

Let $G$ be a connected Lie group. A \textit{Homogeneous space} space $H$ of $G$ is a smooth manifold on which $G$ acts transitively. It can be shown that for any $g \in G$, we have $G/\text{Stab}(g) \cong H$ as differentiable manifolds, where $\text{Stab}(g)$ denotes the \textit{stabilizer subgroup} (also called the \textit{isotropy subgroup}) of $g$. Moreover, for any closed Lie subgroup $K$, the $G$-action $\Phi_g: G/K \to G/K$ satisfying $\Phi_g([h]) = [gh]$ for all $g, h \in G$ is transitive, and so $G/K$ is a homogeneous space. Hence, we may assume without loss of generality that $H := G/K$ is a homogeneous space of $G$ for some closed Lie subgroup $K$. 

Let $\pi: G \to H$ be the canonical projection map. We define the vertical subspace at $g \in G$ by $V_g := \ker (\pi_\ast \vert_g)$, from which we may construct the vertical bundle as $VG := \bigsqcup_{g \in G} \{g\} \times V_g$. Given a Riemannian metric $\left< \cdot, \cdot\right>_G$ on $G$, we may define the horizontal subspace at any point $g \in G$ (with respect to $\left< \cdot, \cdot \right>_G)$ by $\text{Hor}_g := V_g^\perp$, and similarly define the horizontal bundle as $HG := \bigsqcup_{g \in G} \{g\} \times Hor_g$. Both the vertical and horizontal bundles are vector bundles (as discussed in Section \ref{Sec: background}), and are in fact subbundles of the tangent bundle $TG$. It is clear that $T_g G = V_g \oplus \text{Hor}_g$ for all $g \in G$, so that the Lie algebra $\g$ of $G$ admits the decomposition $\g = \mathfrak{s} \oplus \mathfrak{h}$ where $\mathfrak{s}$ is the Lie algebra of $K$ and $\mathfrak{h} \cong T_{\pi(e)} H$. We denote the orthogonal projections onto the vertical and horizontal subspaces by $\mathcal{V}$ and $\mathcal{H}$, respectively. 

%\todo{The notation for the ser of vector fields and the set of sections was not introduced before}

A smooth section $Z \in \Gamma(HG)$ is called a \textit{horizontal vector field}. That is, $Z \in \Gamma(TG)$ and $Z(g) \in \text{Hor}_g$ for all $g \in G.$ A vector field $Y \in \Gamma(TG)$ is said to be $\pi$-related to some $X \in \Gamma(TH)$ if $\pi_\ast Y_g = X_{\pi(g)}$ for all $g \in G$. Given any vector field $X \in \Gamma(TH)$, there exists a unique vector field $\tilde{X} \in \Gamma(HG)$ called the \textit{horizontal lift} of $X$ which is $\pi$-related to $X$. Hence, horizontal lifting provides an injective $\R$-linear map $\tilde{\cdot}: \Gamma(TH) \to \Gamma(HG)$. In general, this map will not be surjective, as it need not be the case that $Z_g = Z_h$ whenever $\pi(g) = \pi(h)$ for $Z \in \Gamma(HG)$. The image of the horizontal lift map will be denoted by $\mathcal{B}(G) \subset \Gamma(HG)$, and its elements will be called \textit{basic vector fields}. That is, $\mathcal{B}(G) \cong \Gamma(TH)$ as an $\R$-vector space, and so a vector field $Z \in \Gamma(TG)$ is basic if and only if it is the horizontal lift of some vector field in $\Gamma(TH)$. Basic vector fields are precisely those which can be pushed-forward to a smooth non-zero vector field on $H$ under $\pi$.

We may also define a horizontal lift of a curve $q: [a, b] \to H$ as a curve $\tilde{q}: [a, b] \to G$ such that $\pi \circ \tilde{q} = q$ and $\dot{\tilde{q}}$ is horizontal. $\tilde{q}$ is not unique in general, but it is unique up to a choice of base point. That is, there is a unique horizontal lift $\tilde{q}$ of $q$ satisfying $\tilde{q}(0) = g_0$, for all $g_0 \in \pi^{-1}(q(0))$. We may similarly call a curve $g: [a, b] \to G$ basic when it is the horizontal lift of some curve $q: [a, b] \to H$, and a vector field $\tilde{X} \in \Gamma(g)$ basic when it is horizontal and $\pi$-related to some $X \in \Gamma(q)$. That is, $(\pi_\ast)_{g(t)} \tilde{X}(t) = X(t)$ for all $t \in [a, b]$. However, it turns out that these distinctions are redundant, as any smooth curve $g: [a, b] \to G$ satisfying $\dot{g}(t) \in \text{Hor}_{g(t)}$ for all $t \in [a, b]$ is necessarily the horizontal lift of some $q: [a, b] \to H$. Namely, it is the horizontal lift of $q := \pi \circ g$. Similarly, every horizontal vector field along a basic curve is basic. To see this, first observe that any curve $\tilde{\eta}: [a, b] \to \mathfrak{h}$ induces a curve $\eta: [a, b] \to T_{\pi(e)}H$ under $(\pi_\ast)_e$. Now notice that the left $G$-action $\Phi_g$ on $H$ satisfies $\Phi_g \circ \pi = \pi \circ L_g$ for all $g \in G$. Taking the differential of both sides, we see that the following diagram commutes:
\begin{equation}\label{commutative_diag}
\begin{tikzcd}
Hor_g \arrow{r}{d\pi_g} \arrow[swap]{d}{L_{g^{-1 \ast}}} & T_{\pi(g)}(G/K) \arrow{d}{\Phi_{g^{-1 \ast}}} \\%
\mathfrak{h} \arrow{r}{d\pi_e}& T_{\pi(e)}(G/K).
\end{tikzcd}  
\end{equation}
Since $\Phi$ is a transitive action, we may translate $\eta$ to a vector field $X$ along any curve $q: [a, b] \to H$ via $X(t) := \Phi_{\tilde{q}(t)\ast}(\eta(t))$ for all $t \in [a, b]$, where $\tilde{q}$ is a horizontal lift of $q$. $X$ may then be horizontally lifted to a unique $\tilde{X}$ along $\tilde{q}$. Then $\tilde{X}$ is basic, and due to the identification \eqref{commutative_diag}, it follows that $L_{(\tilde{q}^{-1})\ast} \tilde{X} \equiv \tilde{\eta}$. This is summarized in the following lemma:

%\begin{lemma}\label{left_hor_vert}
%Let $\xi \in \mathfrak{h}$ and $\sigma \in \mathfrak{s}$. Then $\phi(\xi) \in \Gamma(HG)$ and $\phi(\sigma) \in \Gamma(VG)$.
%\end{lemma}
%\begin{proof}
%First observe that for any $g \in G$, we have 
%\begin{align*}
    %\pi_\ast \phi(\sigma)(g) &= \pi_\ast (L_{g\ast} \sigma) = %\Phi_{g\ast} (\pi_\ast \sigma) = 0,
%\end{align*}
%so that $\phi(\sigma)(g) \in V_g$, and hence $\phi(\sigma) \in \Gamma(VG)$. To see that $\phi(\xi) \in \Gamma(HG)$, observe that
%\begin{align*}
%\left<\phi(\xi)(g), \phi(\sigma)(g)\right>_G = \left<\phi(\xi)(e), \phi(\sigma)(e)\right>_G = \left<\xi, \sigma\right>_{\g} = 0.
%\end{align*}
%\end{proof

%\todo{se hace dificil leer todo el comienzo de la seccion hasta $4.1$ pon ``definitions'' para dar una estructura y una facil lectura} 

%\todo{define Riemannian submersion with words} 

\begin{lemma}\label{lemma: curves_in_h}
Suppose that $q: [a, b] \to H$ and $\tilde{q}: [a, b] \to G$ is a horizontal lift of $q$. Then, for any $\tilde{\eta}: [a, b] \to \mathfrak{h}$, there exists a unique $X \in \Gamma(q)$ such that its horizontal lift $\tilde{X}$ along $\tilde{q}$ satisfies $L_{\tilde{q}(t)^{-1}\ast} \tilde{X}(t) = \tilde{\eta}(t)$ for all $t\in[a,b]$.
\end{lemma}

If a Riemannian metric $\left< \cdot, \cdot \right>_H$ on $H$ can be chosen so that $\pi$ is a Riemannian submersion—that is, so that $\pi_\ast\vert_g$ is a linear isometry between $\text{Hor}_g$ and $T_{\pi(g)} H$ for all $g \in G$—we call $H$ a \textit{Riemannian homogeneous space}. It is clear that in such a case, $\left<\Hor(X), \Hor(Y)\right>_G = \left<\pi_\ast X, \pi_\ast Y \right>_H$ for all $X, Y \in T_g G, g \in G$. In particular, $\left< \tilde{X}, \tilde{Y} \right>_G = \left< X, Y \right>_H$ for all $X, Y \in T_g G, g \in G$. The metric $\left< \cdot, \cdot \right>_H$ is said to be \textit{$G$-invariant} if it is invariant under the left action $\Phi_g$ for all $g \in G$.

\begin{lemma}\label{lemma: G-inv_met}
If $\left< \cdot, \cdot\right>_G$ is left-invariant and $\pi: G \to H$ is a Riemannian submersion, then $\left< \cdot, \cdot \right>_H$ is $G$-invariant.
\end{lemma}
\begin{proof}
Suppose that $q \in H, X, Y \in T_q H, g \in G$, and let $\tilde{X}, \tilde{Y}$ be horizontal lifts of $X, Y$. 
\begin{align*}
    \left<\Phi_{g\ast}X, \Phi_{g\ast}Y\right>_H &= \left<\pi_\ast^{-1} \circ \Phi_{g\ast}X, \pi_\ast^{-1} \circ \Phi_{g\ast}Y\right>_G \\
    &= \left<\pi_\ast^{-1} \circ \Phi_{g\ast} \circ \pi_\ast \tilde{X}, \pi_\ast^{-1} \circ \Phi_{g\ast} \circ \pi_\ast \tilde{Y}\right>_G \\
    &= \left<L_{g\ast} \tilde{X}, L_{g\ast} \tilde{Y} \right>_G
\end{align*}
where the last equality follows from \eqref{commutative_diag}. By the left-invariance of $\left< \cdot, \cdot \right>_G$, we then have that $$\left<L_{g\ast} \tilde{X}, L_{g\ast} \tilde{Y} \right>_G = \left< \tilde{X}, \tilde{Y}\right>_G = \left< X, Y\right>_H.$$
\end{proof}

Denote the Levi-Civita connections on $H$ and $G$ by $\nabla$ and $\tilde{\nabla}$, respectively. The following lemma outlines some useful properties of $\tilde{\nabla}$.

\begin{lemma}\label{lemma: cov_G_to_H}
Let $\tilde{X}, \tilde{Y} \in \mathcal{B}(G)$, and  $X, Y \in \Gamma(TH)$ be the unique vector fields which are $\pi$-related to $\tilde{X}, \tilde{Y}$, respectively. Further suppose that $V \in \Gamma(VG)$. Then, the following identities hold:
\begin{align}
    \tilde{\nabla}_{\tilde{X}}\tilde{Y} &= \widetilde{\nabla_X Y} + \frac12\mathcal{V}([\tilde{X}, \tilde{Y}]), \label{eq: covG-to-covH}\\
    \tilde{\nabla}_{\tilde{X}} V &= \tilde{\nabla}_V \tilde{X}\label{cov_G_commute}
\end{align}
\end{lemma}
\begin{proof}
Equation \eqref{eq: covG-to-covH} follows directly from \cite{zhang2018left}. To see \eqref{cov_G_commute}, observe that $0 = [\tilde{X}, V] = \tilde{\nabla}_{\tilde{X}} V - \tilde{\nabla}_V \tilde{X}$, where the first equality follows from the fact that $\tilde{X}$ is $\pi$-related to $X$, and $V$ is $\pi$-related to the $0$ vector field on $H$, while the second follows from the fact that $\tilde{\nabla}$ is torsion-free.
\end{proof}

\subsection{Reduction on $G/K$ with a Left-Invariant Metric}\label{section: homo_left}
The seven assumptions that we may employ along the remainder of the section are as follows.\\
\begin{quote}
    \textbf{H1} (respectively $\textbf{H1}^\ast$):  $H := G/K$ is a Riemannian homogeneous space, where $G$ satisfies assumption \textbf{G1} (respectively $\textbf{G1}^\ast$). \\
    $\textbf{H1}^{\ast\ast}$: $H := G/K$ is a Riemannian homogeneous space with respect to the left-invariant metric $\left< \cdot, \cdot\right>_G$ on $G$. We denote the levi-civita connection on $G$ with respect to $\left< \cdot, \cdot\right>_G$ by $\tilde{\nabla}$, and the corresponding Riemannian $\g$-connection and curvature tensor are denoted by $\tilde{\nabla}^\g$ and $\tilde{R}$, respectively. $G$ also admits a bi-invariant metric, denoted by $\left< \cdot, \cdot \right>_G^{\text{Bi}}$, and we let $\beta: \g \to \g$ be the linear endomorphism such that $\left<\xi, \eta\right>_G^{\text{Bi}} = \left<\beta(\eta), \eta\right>_G$ for all $\xi, \eta \in \g$. \\
    \textbf{H2:} $\tilde{V}: G \to \R$ defined by $\tilde{V}(g) := V(\pi(g))$ satisfies assumption \textbf{G2}. \\
\end{quote}

%\todo{make a remark to note that some of these seven assumptions are in correspondence with the seven ones given in \cite{bloch2017optimal} and why others are not necessary. Makes more sense in the next section when you have the extended potential, but you can say something here and next another remark in the next section since the 7 hypothesis should be the same}

Let $L := TH \to \R$ be defined by $L(q, X_q) := \left<X_q, X_q\right>_H + V(q)$ and consider the Lagrangian $\tilde{L}: TG \to \R$ defined by $\tilde{L}:= L \circ \pi_\ast$. If we similarly define $\tilde{V} := V \circ \pi$, then it follows that 
\begin{align*}
    \tilde{L}(g, Z_g) &= \|\pi_\ast Z_g\|_H^2 + V(\pi(g)) \\
    &= \|\Hor(Z_g)\|^2_G + \tilde{V}(g),
\end{align*}
for all $g \in G, Z_g \in T_g G$. In particular, if $q: [a, b] \to H$, then 
\begin{align*}
    \tilde{L}(\tilde{q}, \tilde{D}_t \dot{\tilde{q}}) &= \|\Hor(\tilde{D}_t \dot{\tilde{q}})\|^2_G + \tilde{V}(\tilde{q}) \\
    &= \|\widetilde{D_t \dot{q}}\|^2_G + V(\pi \circ \tilde{q}) \\
    &= \|D_t \dot{q}\|^2_H + V(q) \\
    &= L(q, D_t \dot{q}).
\end{align*}
In other words, for any horizontal lift $g$ of $q$,
$$\int_a^b \left(\|\Hor(\tilde{D}_t \dot{g}) \|^2_G + \tilde{V}(g)\right)dt = \int_a^b \left(\|D_t \dot{q}\|^2_H + V(q)\right)dt.$$
This motivates the definition of the operator $\HorD_t X := \Hor(\tilde{D}_t X)$ for all $X \in \Gamma(g)$. It is clear that we have $\HorD_t X = \HorLC_{\dot{g}} X$, where $\HorLC: \Gamma(TG) \times \Gamma(HG) \to \Gamma(HG)$ is the connection on the horizontal bundle defined by $\HorLC_W Z = \mathcal{H}(\tilde{\nabla}_W Z)$ for all $W \in \Gamma(TG), \ Z \in \Gamma(HG)$.

\begin{theorem}\label{thm: variational_principles_homo_G}
Consider a curve $q: [a, b] \to G$ and let $g$ be a horizontal lift of $q$. Then $q$ is a modified Riemannian cubic with respect to $V$ if and only if $g$ satisfies the variational principle
\begin{equation}\label{G_variational_prob}
    \delta\int_a^b \left(\|\HorD_t \dot{g}\|^2_G + \tilde{V}(g)\right)dt = 0
\end{equation}
among all basic variations of $g$, where $\tilde{V} = V \circ \pi$.
\end{theorem}

\begin{proof}
It is clear that the projection of any basic variation of $g$ is an admissible proper variation of $q$. On the other hand, any admissible proper variation of $q$ is clearly the projection of some basic variation of $g$. Suppose that $q$ is a modified cubic polynomial with respect to $V$ and let $g_s$ is a basic variation of $g$. Then $q_s := \pi \circ g_s$ is an admissible proper variation of $q$, and thus
$$\frac{d}{ds}\Big{|}_{s = 0}\int_a^b \left(\|\HorD_t \dot{g}_s\|^2_G + \tilde{V}(g_s)\right)dt = \frac{d}{ds}\Big{|}_{s = 0}\int_a^b \left(\|D_t \dot{q}_s\|^2_H + V(q_s) \right)dt = 0.$$
Hence \eqref{G_variational_prob} is satisfied among all basic variations of $g$. Now suppose that $g$ solves \eqref{G_variational_prob} among all basic variations. Then for any admissible proper variation $q_s$ of $q$, we have
$$\frac{d}{ds}\Big{|}_{s = 0}\int_a^b \left(\|D_t \dot{q}_s\|^2_H + V(q_s) \right)dt = \frac{d}{ds}\Big{|}_{s = 0}\int_a^b \left(\|\HorD_t \dot{\tilde{q}}_s\|^2_G + \tilde{V}(\tilde{q}_s)\right)dt = 0,$$
so that $q$ is a modified Riemannian cubic with respect to $V$.
\end{proof}

The remainder of the section is dedicated to deriving necessary conditions for the variational principle \eqref{G_variational_prob}, and then using the symmetry of $G$ and partial symmetry of $V$ to reduce them to some set of equations on $\mathfrak{h}$. Before this, we must study the properties of the horizontal connection. Observe that while $\HorLC$ is only a connection in the strict sense on the domain $\Gamma(TG) \times \Gamma(HG)$, it is still a well-defined operation on the larger domain $\Gamma(TG) \times \Gamma(TG)$. Here we outline some additional properties satisfied by $\HorLC$ and basic vector fields:

\begin{lemma}\label{lemma: Hor_G_to_H}
Let $\tilde{X}, \tilde{Y} \in \mathcal{B}(G)$, and  $X, Y \in \Gamma(TH)$ be the unique horizontal vector fields which are $\pi$-related to $\tilde{X}, \tilde{Y}$, respectively. Further suppose that $W, Z \in \Gamma(HG)$ and $P \in \Gamma(TG)$. Then, the following identities hold:
\begin{align}
    \HorLC_{\tilde{X}} \tilde{Y} &= \widetilde{\nabla_X Y}\label{nabla_lift} \\
    \Hor([\tilde{X}, \tilde{Y}]) &= \widetilde{[X, Y]}\label{Lie_lift} \\
    P\left< W, Z\right>_G &= \left<\HorLC_P W, Z \right>_G + \left< W, \HorLC_{P}Z \right>_G\label{compatibility} \\
    \HorLC_{W} Z - \HorLC_{Z} W &= \Hor([W, Z])\label{torsion}
\end{align}
\end{lemma}
\begin{proof}
Observe that \eqref{nabla_lift} follows immediately from the definition of $\HorLC$ and Lemma \ref{lemma: cov_G_to_H}. \eqref{Lie_lift} follows from the fact that basic vector fields can be pushed forward by $\pi$, hence $\pi_\ast [\tilde{X}, \tilde{Y}] = [\pi_\ast \tilde{X}, \pi_\ast \tilde{Y}] = [X, Y]$, from which the conclusion immediately follows. To see equation \eqref{compatibility}, note that the metric compatibility of $\tilde{\nabla}$ implies that $P\left< W, Z\right>_G = \left<\tilde{\nabla}_P W, Z \right>_G + \left< W, \tilde{\nabla}_{P}Z \right>_G$. Now since horizontal and vertical vectors are orthogonal with respect to $\left< \cdot, \cdot \right>_G$, it follows that $\left<\tilde{\nabla}_P W, Z \right>_G + \left< W, \tilde{\nabla}_{P}Z \right>_G = \left<\Hor(\tilde{\nabla}_P W), Z \right>_G + \left< W, \Hor(\tilde{\nabla}_{P}Z) \right>_G$, from which the conclusion follows. Finally, \eqref{torsion} follows from the fact that $\tilde{\nabla}$ is torsion-free.
\end{proof}

We now define the type $(1, 2)$-tensor field $A: \mathfrak{X}(G) \times \mathfrak{X}(G) \to \mathfrak{X}(G)$ by 
\begin{equation}\label{A_tensor}
    A_X Y = \mathcal{H}\left(\tilde{\nabla}_{\mathcal{H}(X)} \mathcal{V}(Y)\right) + \mathcal{V}\left(\tilde{\nabla}_{\mathcal{H}(X)} \mathcal{H}(Y)\right),
\end{equation}

The following lemma follows immediately from \cite{Oneill}.
\begin{lemma}
Let $X, Y \in B(G)$ and $V \in \Gamma(VG)$. Then,
\begin{align}
    A_X Y &= \frac12 \mathcal{V}([X, Y])\label{A_X Y} \\
    A_X V &= \HorLC_X V \label{A_X V}
\end{align}
\end{lemma}

In order to relate the second derivatives with respect to $\HorLC$, we now study the curvature endomorphism $Q: \Gamma(TG) \times \Gamma(TG) \times \Gamma(HG) \to \Gamma(HG)$ correspinding to $\HorLC$. Namely, we have
\begin{equation}\label{Q}
    Q(W, X)Y = \HorLC_W \HorLC_X Y -  \HorLC_X \HorLC_W Y - \HorLC_{[W, X]} Y 
\end{equation}
and similarly, the horizontal curvature tensor $\text{Qm}$ by $Qm(W, X, Y, Z) := \left<Q(W, X)Y, Z\right>_G$. Note that, due to the fact that $[W, X]$ is not in general basic even when $W, X$ are, the symmetries of $\text{Qm}$ don't immediately follow from the symmetries of $\text{Rm}$ or $\widetilde{\text{Rm}}$. In fact, if $\tilde{W}, \tilde{X}, \tilde{Y} \in B(G)$ are the horizontal lifts of $W, X, Y \in \Gamma(TH)$, then we have
\begin{align*}
    Q(\tilde{W}, \tilde{X}) \tilde{Y} &= \HorLC_{\tilde{W}} \HorLC_{\tilde{X}} {\tilde{Y}} -  \HorLC_{\tilde{X}} \HorLC_{\tilde{W}} {\tilde{Y}} - \HorLC_{[{\tilde{W}}, {\tilde{X}}]} {\tilde{Y}} \\
    &= \widetilde{\nabla_W \nabla_X Y} - \widetilde{\nabla_X \nabla_W Y} - \widetilde{\nabla_{[W, X]} Y} - \HorLC_{\mathcal{V}([\tilde{W}, \tilde{X}])} \tilde{Y} \\
    &= \widetilde{R(W, X)Y} - 2A_{\tilde{Y}} A_{\tilde{W}} \tilde{X},
\end{align*}
where we have used \eqref{A_X V} together with \eqref{cov_G_commute} in the last equality. Hence, it suffices to study the symmetries of the $4$-tensor field $\text{Am}: \Gamma(TG) \times \Gamma(TG) \times \Gamma(TG) \times \Gamma(TG) \to \R$ defined by $\text{Am}(W, X, Y, Z) = \left<A_W A_X Y, Z\right>_G$ along $B(G)$ in order to understand the symmetries of $Qm$ along $B(G)$. We also define the $(3, 1)$-tensor field $\tilde{Q}$ defined by $\tilde{Q}(\tilde{W}, \tilde{X})\tilde{Y} = Q(\tilde{W}, \tilde{X})\tilde{Y} + 2 A_{\tilde{Y}} A_{\tilde{W}}\tilde{X}$, from which it is clear that $\tilde{Q}$ maps basic vector fields to basic vector fields. In particular, $\tilde{Q}$ it is the horizontal lift of the Riemannian curvature endomorphism $R$ on $H$. The symmetries of Am are summarized in the following lemma:

\begin{lemma}\label{Am_symmetries}
Let $W, X, Y, Z \in B(G)$. Then,
\begin{align}
    Am(W, X, Y, Z) &= -Am(W, Y, X, Z)\label{Am23} \\
    Am(W, X, Y, Z) &= -Am(Z, X, Y, W)\label{Am14} \\
    Am(W, X, Y, Z) &= Am(X, W, Z, Y)\label{Am1234}
\end{align}
\end{lemma}
\begin{proof}
\eqref{Am23} follows immediately from the \eqref{A_X Y}. To see \eqref{Am14} and \eqref{Am1234}, first observe that $\left< A_X Y, Z \right>_G = 0$. By the metric compatibility of $\tilde{\nabla}$, it then follows that $0 = W\left< A_X Y, Z \right>_G = \left< \HorLC_W A_X Y, Z \right>_G + \left< A_X Y, \tilde{\nabla}_W Z \right>_G$. It is clear from \eqref{A_X Y} that $A_X Y$ is vertical, hence \eqref{A_X V} implies that $\HorLC_W A_X Y = A_W A_X Y$. Moreover, \eqref{A_tensor} shows that $\left< A_X Y, \tilde{\nabla}_W Z \right>_G = \left< A_X Y, A_W Z \right>_G$. Hence, $\left<A_W A_X Y, Z\right>_G = -\left<A_X Y, A_W Z\right>_G$. \eqref{Am14} now follows as:
\begin{align*}
    Am(W, X, Y, Z) &= -\left<A_X Y, A_W Z\right>_G \\
    &= \left<A_X Y, A_Z W \right>_G \\
    &= -\left<A_Z A_X Y, W \right>_G \\
    &= -Am(Z, X, Y, W).
\end{align*}
Similarly, \eqref{Am1234} can be seen from:
\begin{align*}
    Am(W, X, Y, Z) &= -\left<A_X Y, A_W Z\right>_G \\
    &= -\left<A_W Z, A_X Y\right>_G \\
    &= \left<A_X A_W Z, Y\right>_G \\
    &= Am(X, W, Z, Y).
\end{align*}
\end{proof}

We are now in a position to study the symmetry relations of Qm on $B(G)$. As we will see, despite the dependence on $Am$, most of the symmetries of $Rm$ will be preserved. The only exception is that Qm$(W, X, Y, Z) + \text{Qm}(X, Y, W, Z) + \text{Qm}(Y, W, X, Z)$ in general fails to vanish. 
\begin{lemma}\label{Qm_symmetries}
For all $W, X, Y, Z \in B(G)$, the following relations hold:
\begin{align}
    \text{Qm}(W, X, Y, Z) &= -\text{Qm}(X, W, Y, Z) \label{Qm12} \\
    \text{Qm}(W, X, Y, Z) &= -\text{Qm}(W, X, Z, Y) \label{Qm34}\\
    \text{Qm}(W, X, Y, Z) &= \text{Qm}(Y, Z, W, X) \label{Qm1324} 
\end{align}
\end{lemma}
\begin{proof}
\eqref{Qm12} follows immediately from the definition of Qm. Now let $\bar{W} := \pi_\ast W$ and similarly for $\bar{X}$ and $\bar{Y}$. Then, we have:
\begin{align*}
    \text{Qm}(W, X, Y, Z) &= \widetilde{\text{Rm}(\bar{W}, \bar{X}, \bar{Y}, \bar{Z})} + 2\text{Am}(Y, X, W, Z) \\
    \text{Qm}(W, X, Z, Y) &= \widetilde{\text{Rm}(\bar{W}, \bar{X}, \bar{Z}, \bar{Y})} + 2\text{Am}(Z, X, W, Y) \\
    &= -\widetilde{\text{Rm}(\bar{W}, \bar{X}, \bar{Y}, \bar{Z})} - 2\text{Am}(Y, X, W, Z)
\end{align*}
By using the symmetries of Rm and \eqref{Am14}. Adding the two equations yields \eqref{Qm34}. Similarly,
\begin{align*}
    \text{Qm}(W, X, Y, Z) &= \widetilde{\text{Rm}(\bar{W}, \bar{X}, \bar{Y}, \bar{Z})} + 2\text{Am}(Y, X, W, Z) \\
    \text{Qm}(Y, Z, W, X) &= \widetilde{\text{Rm}(\bar{Y}, \bar{Z}, \bar{W}, \bar{X})} + 2\text{Am}(Z, X, W, Y) \\
    &= \widetilde{\text{Rm}(\bar{W}, \bar{X}, \bar{Y}, \bar{Z})} + 2\text{Am}(Y, X, W, Z),
\end{align*}
Since applying \eqref{Am23}-\eqref{Am1234} each one time to Am$(Z, X, W, Y)$ yields Am$(Y, X, W, Z)$. Subtracting these equations from each other yields \eqref{Qm1324}.
\end{proof}

We now return to vector fields along curves. Since $A$ is tensorial, we may also evaluate it along a vector field along some curve $g: [a, b] \to G$ by defining $A_t X = A_{\dot{g}} X$ for all $X \in \Gamma(g)$. Moreover, if $\xi = g^{-1} \dot{g}$ and $\eta = g^{-1} X$, then we have $A_t X = g A_{\xi} \eta$. Note, however, that the right-hand side of \eqref{A_tensor} is only well-defined in the case that $\dot{g}$ is horizontal, at which point it is clear that $A_t X = \HorD_t \mathcal{V}(X) + \mathcal{V}(\tilde{D}_t \Hor(X))$. To evaluate this quantity when $g$ is not horizontal, we need only consider for each $\tau \in [a, b]$ any vector fields $Z_{\tau}, Y_{\tau} \in \Gamma(TG)$ such that $Z_{\tau}(g(\tau)) = \dot{g}(\tau)$ and $Y_{\tau} = X(\tau)$, and then evaluate $(A_t X)(\tau) = (A_{Z_{\tau}} Y_{\tau})(g(\tau))$.

We now seek to derive the necessary conditions corresponding to \eqref{G_variational_prob}. We expect to obtain equations resembling \eqref{eqq1}, since the properties of $\HorD$ and $Q$ mirror those of $D_t$ and $R$. Before this, we must calculate the commutativity of covariant derivatives along basic variations.

\begin{lemma} Let $g: [a, b] \to G$ be a basic curve and $g_s$ be a basic variation of $g$. Define $T = \partial_t g_s$ and $S = \partial_s g_s$. Then,
\begin{equation}
    \HorD_s \HorD_t T - \HorD_t \HorD_s T = \tilde{Q}(S, T)T.
\end{equation}
\end{lemma}
\begin{proof}
Since $g$ and $g_s$ are basic, there exists a curve $q: [a, b] \to H$ and a corresponding proper variation $q_s$ such that $g, g_s$ are horizontal lifts of $q, q_s$, respectively. If we let $\bar{T} = \partial_t q_s$ and $\bar{S} = \partial_s q_s$, then it is clear that $T, S$ are the horizontal lifts of $\bar{T}, \bar{S}$, respectively. Hence,
\begin{align*}
    \HorD_s \HorD_t T - \HorD_t \HorD_s T &= \widetilde{D_s D_t \bar{T}} - \widetilde{D_t D_s \bar{T}} \\
    &= \widetilde{R(\bar{S}, \bar{T})\bar{T}} \\
    &= \tilde{Q}(S, T)T
\end{align*}
\end{proof}
We now derive the necessary conditions for optimality in the variational principle \eqref{G_variational_prob}.
\begin{proposition}\label{thm: necessary_homo_G}
For any smooth artificial potential $\tilde{V}: G \to \R$, a basic curve $g: [a, b] \to G$ satisfies the variational principle \eqref{G_variational_prob} if and only if it is smooth and satisfies
\begin{equation}\label{lift_cubic}
    \left(\HorD_t\right)^3 \dot{g} + \tilde{Q}\left(\HorD_t \dot{g}, \dot{g}\right)\dot{g} + \Hor(\grad \tilde{V}(g))= 0
\end{equation}
on the full interval $[a, b].$
\end{proposition}
\begin{proof}
Consider a basic variation $g_s$ of $g$, and let $T = \partial_t g_s$ and $S = \partial_s g_s$. Further denote $\delta g = S\vert_{s=0}$ Then,
\begin{align*}
    0 &= \frac{d}{ds}\Big\vert_{s=0} \int_a^b \left(\left<\HorD_t \dot{g}_s, \HorD_t \dot{g}_s\right>_G + \tilde{V}(g_s)\right)dt \\
    &= 2\int_a^b \left(\left<\HorD_s \HorD_t T, \HorD_t T\right>_G + \left<\grad\tilde{V}(g_s), S\right>_G\right)\Big\vert_{s=0}dt \\
    &= 2\int_a^b \left(\left<\HorD_t \HorD_s T, \HorD_t T\right>_G + \left<\tilde{Q}(S, T) T, \HorD_t T\right>_G + \left<\grad\tilde{V}(g_s), S\right>_G\right)\Big\vert_{s=0}dt \\
    &= 2\int_a^b \left(\left<\HorD_t \HorD_t S, \HorD_t T\right>_G + \left<\tilde{Q}(\HorD_t T, T)T + \grad\tilde{V}(g_s),  S\right>_G\right)\Big\vert_{s=0}dt \\
    &=2\int_a^b \left<\left(\HorD_t\right)^3 \dot{g} + \tilde{Q}(\HorD_t \dot{g}, \dot{g})\dot{g} + \Hor(\grad \tilde{V}(g)), \delta g\right>_G dt,
\end{align*}
where in the last line, we applied integration by parts twice to the first term and evaluated at $s = 0$, and used the fact that $\left<X, Y\right>_G = \left<\Hor(X), Y\right>_G$ for all $X \in \Gamma(TG), Y \in \Gamma(HG)$. As a consequence of Lemma \ref{lemma: curves_in_h} and the definitions of $\HorD$ and $\tilde{Q}$, the left-hand side of \eqref{lift_cubic} is a basic vector field. The conclusion follows upon setting $\delta g = \left(\HorD_t\right)^3 \dot{g} + \tilde{Q}(\HorD_t \dot{g}, \dot{g})\dot{g} + \grad \tilde{V}(g)$.
\end{proof}

Note that, since $g$ is a basic curve, it is the horizontal lift of some $q: [a, b] \to H$ (in particular, $q = \pi \circ g$). It is clear from their respective definitions that $\left(\HorD_t\right)^3 \dot{g} = \widetilde{D_t^3 \dot{q}}$ and $\tilde{Q}(\HorD_t \dot{g}, \dot{g})\dot{g} = \widetilde{R(D_t \dot{q}, \dot{q})\dot{q}}$. Moreover, from Lemma \ref{lemma: curves_in_h}, there exists some vector field $X \in \Gamma(q)$ whose tangent lift satisfies $\tilde{X}(t) = \Hor( \grad \tilde{V}(g(t)))$ for all $t \in [a,b]$. Observe that $X$ is locally extendible. That is, for any $t_0 \in [a, b]$, there exists some subinterval $(a^\ast, b^\ast) \subset [a, b]$ containing $t_0$, a neighborhood $U \subset H$ containing $q((a^\ast, b^\ast)),$ and a smooth vector field $Y: U \to TH$ such that $Y(q(t)) = X(t)$ for all $t \in (a^\ast, b^\ast).$ Consequently, the horizontal lift of $Y$ satisfies $Y(g(t)) = \Hor(\grad \tilde{V}(g(t))$ for all $t \in (a^\ast, b^\ast)$. On the other hand, consider the $1$-form $Y^\flat$. If $H$ is simply connected (this follows for instance if $G$ is simply connected and $K$ is connected), then there must exist a scalar field $V: U \to \R$ such that $dV = Y^\flat$. Hence, $\grad V = Y$, and so $\widetilde{\grad V}(g(t)) = \Hor(\tilde{V}(g(t)))$ for all $t \in (a^\ast, b^\ast)$. Observe that these vector fields need not agree away from curve $g$, however we obtain immediately that $g$ is locally the horizontal lift of a modified cubic with respect to $V$. This, together with Theorem \ref{thm: variational_principles_homo_G}, leads to the following corollary to Proposition \ref{thm: necessary_homo_G}.

\begin{corollary}\label{Cor: arbitrary_potential}
Suppose that $g: [a, b] \to G$ is a basic curve satisfying equation \eqref{lift_cubic} for some smooth $\tilde{V}: G \to \R$. Then for each $t_0 \in [a, b]$, there exists some $(a^\ast, b^\ast) \subset [a, b]$ containing $t_0$, a neighborhood $U \subset H$ containing $q((a^\ast, b^\ast)),$ and a scalar field $V: U \to \R$ such that the curve $q \vert_{(a^\ast, b^\ast)} := \pi \circ g \vert_{(a^\ast, b^\ast)}$ is a modified cubic polynomial with respect to $V$. Moreover, $q$ is a modified cubic with respect to $V$ on [a, b] if and only if $\tilde{V} = V \circ \pi$ for some smooth $V: H \to \R$. 
\end{corollary}

In the case that $\tilde{V} = V \circ \pi$ for some $V: H \to \R$, observe that $\grad \tilde{V}(g)$ must be horizontal since $\left<\grad \tilde{V}(h), U\right>_G = d\tilde{V}(U) = dV \circ \pi_\ast(U) = 0$ for all $h \in G, U \in V_h$. Hence $\Hor(\grad \tilde{V}) = \grad \tilde{V}$. Moreover, if we consider any $\tilde{X} \in Hor_h$, then:
\begin{align*}
    \left<\grad \tilde{V}(h), \tilde{X}\right>_G = d\tilde{V}(\tilde{X}) = dV(X) = \left<\grad V(\pi(h)), X\right>_H = \left<\widetilde{\grad V(\pi(h))}, \tilde{X}\right>_G,
\end{align*}
where $X = \pi_\ast \tilde{X}$. Therefore, $\grad \tilde{V}$ is precisely the horizontal lift of $\grad V$. It follows that that equation \eqref{lift_cubic} is precisely the horizontal lift of equation \eqref{eqq1} on $H$. Horizontally lifting \eqref{eqq1} is a strategy that was employed in \cite{zhang2018left} in the case that $V \equiv 0$. In particular, equations were provided in terms of the Levi-Civita connection $\tilde{\nabla}$ and the Riemannian curvature $\tilde{R}$ on $G$. However, the form of these equations did not resemble equation \eqref{eqq1} describing Riemannian cubics on $G$, whereas \eqref{lift_cubic} takes the same form as \eqref{eqq1}, with $\tilde{D}_t$ replaced by $\HorD_t$ and $\tilde{R}$ replaced by $\tilde{Q}$. Hence we see that the basic modified cubics satisfy a direct anologue of the equation describing modified cubics. We need only consider the projection of the Levi-Civita connection onto the horizontal bundle, and its corresponding curvature tensor. Moreover, we were able to find local horizontal lifts in Corollary \ref{Cor: arbitrary_potential}, which allows for more diverse artificial potentials on $G$ (this will become important in Section \ref{Sec: homo_bi}). The most substantial benefit to working directly with the variational principles instead of lifting the resulting equations to $G$, however, is that we may adapt the situation to reduction by symmetry in the case of symmetry breaking artificial potentials (as in Section \ref{reduction_obs_avoid_G}). The remainder of the section is dedicated to finding the Euler-Poincar\'e equations corresponding to \eqref{lift_cubic}.

As with \eqref{g-connection}, we may define the Riemannian $\mathfrak{h}$-connection $\tilde{\nabla}^{\mathfrak{h}}: \mathfrak{h} \times \mathfrak{h} \to \mathfrak{h}$ via 
\begin{equation}
    \hcon_{\xi}\eta = \left(\HorLC_{\phi(\xi)}\phi(\eta)\right)(e).
\end{equation}
It is clear that $\hcon_\xi \eta = \Hor(\tilde{\nabla}_{\xi}^\g\eta)$, where $\tilde{\nabla}^\g$ is the Riemannian $\g$-connection corresponding to the Levi-Civita connection $\tilde{\nabla}$ on $G$. Therefore, we obtain the explicit expression 
\begin{equation}\label{hcon_decomp}
    \hcon_{\xi} \eta =\frac12 \Hor([\xi, \eta]_\g - \ad^\dagger_{\xi} \eta - \ad^\dagger_\eta \xi).
\end{equation}
We now express Lemma \ref{lemma: cov-to-covg} in terms of the Riemannian $\mathfrak{h}$-connection and horizontal connection:

\begin{lemma}\label{lemma: Horcov-to-covh}
Let $g: [a,b] \to G$ be a basic curve and $X$ a smooth horizontal vector field along $g$. Suppose that $\xi(t) = g(t)^{-1} \dot{g}(t)$ and $\eta(t) = g(t)^{-1} X(t)$. Then the following relation holds for all $t \in [a, b]$:
\begin{align}
    \HorD_t X(t) = g(t)\left(\dot{X}(t) + \nabla_{\xi}^{\mathfrak{h}} \eta(t) \right).
\end{align}
\end{lemma}
Proposition \ref{prop: cubic_red} follows analogously, which leads to the following Euler-Poincar\'e equations corresponding to \eqref{lift_cubic}:

\begin{proposition}\label{prop: reduction_left_inv_H}
Consider the variational obstacle avoidance problem \textbf{P1} with $Q = H$ and the additional assumptions \textbf{H1, H2}. Suppose that $q \in \Omega$ and let $g$ be a horizontal lift of $q$. Then $q$ is a modified cubic polynomial with respect to $V$ if and only if $\xi := g^{-1} \dot{g}$ and $h := g_0^{-1}g$ satisfy:
\begin{align}
    \dot{\xi} &= \eta + \Hor(\ad^\dagger_{\xi} \xi), \label{ad_dag_h}\\
    \ddot{\eta} + 2\nabla^{\mathfrak{h}}_\xi \dot{\eta} + \nabla^{\mathfrak{h}}_{\eta} \eta + \nabla^{\mathfrak{h}}_{\Hor(\ad^\dagger_\xi \xi)} \eta +& \nabla^{\mathfrak{h}}_\xi \nabla^{\mathfrak{h}}_\xi \eta + \tilde{Q}\big{(}\eta, \xi \big{)}\xi + L_{h^{-1 \ast}} \grad_1 \tilde{V}_{\ext}(h, e) = 0, \label{eqq2_h} \\
    \dot{h}(t) &= h(t)\xi(t). \label{eqqh_h}
\end{align}
\end{proposition}

\subsection{Reduction on Symmetric Spaces and Homogeneous Spaces with bi-invariant Metrics}\label{Sec: homo_bi}

Analogously to Section \ref{sec: red_biinv}, we are interesting further simplifying equations \eqref{ad_dag_h} - \eqref{eqqh_h} in the case that $G$ is endowed with a bi-invariant metric. We first obtain simplified expressions for the the Riemannian $\mathfrak{h}$-connection and the tensor field $\tilde{Q}$. 

\begin{lemma}\label{h_bi_identities}
Suppose that $\left< \cdot, \cdot\right>_G$ is bi-invariant, and let $\xi, \eta, \sigma \in \mathfrak{h}$. Then,
\begin{align}
    \hcon_{\xi}\eta &= \frac12 \Hor\big([\xi, \eta]\big), \label{nabla_h bi}\\
    \tilde{Q}(\xi, \eta)\sigma &= \frac14\big(\Hor\big([\sigma, [\xi, \eta]]\big) - [\xi, \mathcal{V}([\eta, \sigma]) + [\eta, \mathcal{V}([\xi, \sigma])] + 2[\sigma, \mathcal{V}([\xi, \eta])]\big). \label{Q_tild_bi}
\end{align}
\end{lemma}

\begin{proof}
\eqref{nabla_h bi} follows immediately from Lemma \ref{bi-invariant_G} together with the fact that $\hcon_\xi \eta = \Hor(\tilde{\nabla}^\g \xi \eta)$. For \eqref{Q_tild_bi},it follows from Theorem $2.3$ of \cite{zhang2018left} that $\tilde{Q}(\xi, \eta)\sigma = \Hor(\tilde{R}(\xi, \eta)\sigma) - A_{\xi} A_{\eta} \sigma + A_{\eta} A_{\xi} \sigma + 2 A_{\sigma} A_{\xi} \eta$. Note that $\Hor(\tilde{R}(\xi, \eta)\sigma) = \frac14 \Hor([\sigma, [\xi, \eta]])$ from Lemma \ref{bi-invariant_G}. From the definition of $A$, it is clear that if $u_1, u_2 \in \mathfrak{h}$ and $v \in \mathfrak{s}$, then $A_{u_1} u_2 = \mathcal{V}(\tilde{\nabla}^\g_{u_1} u_2) = \frac12 \mathcal{V}([u_1, u_2])$ and $A_{u_1} v = \hcon_{u_1} v = \frac12\Hor([u_1, v])$. Hence, $A_{\xi} A_{\eta} \sigma = \frac14 \Hor( [\xi, \mathcal{V}([\eta, \sigma])])$, and similarly for $A_{\eta} A_{\xi} \sigma$ and $A_{\sigma} A_{\xi} \eta.$ Finally, we show that $G/K$ is in fact a \textit{reductive} homogeneous space. That is, $[\mathfrak{h}, \mathfrak{s}] \subset \mathfrak{h}$. Equation \eqref{Q_tild_bi} follows immediately upon observing that $\mathcal{V}: \g \to \mathfrak{s}$, and $\mathcal{H}$ acts as the identity map on $\mathfrak{h}$. To see that $G/K$ is reductive, first observe that $[\mathfrak{s}, \mathfrak{s}] \subset \mathfrak{s}$ since $K$ is a Lie subgroup of $G$, and hence $\mathfrak{s}$ is a Lie subalgebra of $\g$. Now suppose that $x, y \in \mathfrak{s}, z \in \mathfrak{h}$. Then, $0 = \left< x, z \right> = \left< \Ad_{\Exp(ty)} x, \Ad_{\Exp(ty)} z \right>.$ Taking a derivative at $t=0$ then yields $0 = \left<[y, x], z\right> + \left<x, [y, z]\right> = \left<x, [y, z]\right>.$ Hence $[y, z] \in \mathfrak{h}$, and the conclusion follows. 
Combining these yields \eqref{Q_tild_bi}.
\end{proof}

We then obtain the following Proposition by combining Proposition \ref{prop: reduction_left_inv_H} with Lemma \ref{h_bi_identities}.

\begin{proposition}\label{prop: reduction_bi_inv_H}
Consider the variational obstacle avoidance problem \textbf{P1} with $Q = H$ and the additional assumptions $\textbf{H1}^\ast$, \textbf{H2}. Suppose that $q \in \Omega$ and let $g$ be a horizontal lift of $q$. Then $q$ is a modified cubic polynomial with respect to $V$ if and only if $\xi := g^{-1} \dot{g}$ and $h := g_0^{-1}g$ satisfy:
\begin{align}
    \dddot{\xi} + \Hor([\xi, \ddot{\xi}]) + [\xi, [\dot{\xi}, \xi]] - \frac34 &[\xi, \Hor([\dot{\xi}, \xi])] + L_{h^{-1 \ast}} \grad_1 \tilde{V}_{\ext}(h, e) = 0, \label{eqq2_h_bi} \\
    \dot{h}(t) &= h(t)\xi(t). \label{eqqh_h_bi}
\end{align}
Similarly, under the assumptions $\textbf{H1}^{\ast \ast}$, \textbf{H2}, $q$ is a modified cubic polynomial with respect to $V$ if and only if $\xi$ and $h$ satisfy:
\begin{align}
    \dddot{\xi} + \Hor([\xi, \ddot{\xi}]) + [\xi, [\dot{\xi}, \xi]] - \frac34 &[\xi, \Hor([\dot{\xi}, \xi])] + \beta(L_{h^{-1 \ast}} \grad_1^{\text{Bi}} \tilde{V}_{\ext}(h, e)) = 0, \label{eqq2_h_bi_comp} \\
    \dot{h}(t) &= h(t)\xi(t). \label{eqqh_h_bi_comp}
\end{align}
\end{proposition}

\begin{proof}
First observe that, since the metric is bi-invariant, we have $\left<\ad_x y, z\right> + \left<y, \ad_x z\right> = 0$ for all $x, y, z \in \g$. Hence $\ad^\dagger = -\ad$, and since $\ad$ is skew-symmetric, we obtain $\ad_{\xi}^\dagger \xi = 0$, so that \eqref{ad_dag_h} becomes $\dot{\xi} = \eta$. Equation \eqref{eqq2_h_bi} then follows directly from Lemma \ref{h_bi_identities} together with the decomposition $x = \mathcal{H}(x) + \mathcal{V}(x)$ for all $x \in \g$. Equation \eqref{eqq2_h_bi_comp} follows similarly with the proof strategy given  in \ref{prop: reduction_left_inv_comp}.
\end{proof}

Another special class of spaces that appears frequently in applications is \textit{Riemannian symmetric spaces}. These are Riemannian homogeneous spaces such that there exists an involutive automorphism $\sigma: G \to G$ with $K = \{g \in G \: \ \sigma(g) = g\}.$ It can be seen that every Riemannian symmetric space of the form $G/K$ satisfies the Cartan Decomposition:
\begin{equation}\label{cartan}
    [\mathfrak{s}, \mathfrak{s}] \subset \mathfrak{s}, \quad [\mathfrak{s}, \mathfrak{h}] \subset \mathfrak{h}, \quad [\mathfrak{h}, \mathfrak{h}] \subset \mathfrak{s}.
\end{equation}
In particular, every Riemannian symmetric space is reductive. Moreover, if $G$ is simply connected, then every Riemannian homogeneous space satisfying \eqref{cartan} is a Riemannian symmetric space. It is neither necessary nor sufficient that $\left< \cdot, \cdot\right>_G$ be bi-invariant in order for $G/K$ to be Riemannian symmetric. Hence, we have the following additional assumptions:

\begin{quote}
    \textbf{S1} (respectively $\textbf{S1}^\ast$, $\textbf{S1}^{\ast\ast}$):  $H := G/K$ is a Riemannian symmetric space, where $G$ satisfies assumption \textbf{G1} (respectively $\textbf{G1}^\ast$, $\textbf{G1}^{\ast\ast}$). \\
\end{quote}

Using equation \eqref{cartan}, we obtain the following Propositions corresponding to Proposition \ref{prop: reduction_left_inv_H}:

\begin{proposition}\label{prop: reduction_sym_left}
Consider the variational obstacle avoidance problem \textbf{P1} with $Q = H$ and the additional assumptions $\textbf{S1}$, \textbf{H2}. Suppose that $q \in \Omega$ and let $g$ be a horizontal lift of $q$. Then $q$ is a modified cubic polynomial with respect to $V$ if and only if $\xi := g^{-1} \dot{g}$ and $h := g_0^{-1}g$ satisfy:
\begin{align}
    \dddot{\xi} + \tilde{Q}(\eta, \xi)\xi + &L_{h^{-1 \ast}} \grad_1 \tilde{V}_{\ext}(h, e) = 0, \label{eqq2_sym} \\
    \dot{h}(t) &= h(t)\xi(t). \label{eqqh_sym}
\end{align}
Similarly, under the assumptions $\textbf{S1}^{\ast}$, \textbf{H2}, $q$ is a modified cubic polynomial with respect to $V$ if and only if $\xi$ and $h$ satisfy:
\begin{align}
    \dddot{\xi} + [\xi, [\dot{\xi}, \xi]] + &L_{h^{-1 \ast}} \grad_1 \tilde{V}_{\ext}(h, e) = 0, \label{eqq2_sym_bi} \\
    \dot{h}(t) &= h(t)\xi(t). \label{eqqh_sym_bi}
\end{align}
Finally, under the assumptions $\textbf{S1}^{\ast \ast}$, \textbf{H2}, $q$ is a modified cubic polynomial with respect to $V$ if and only if $\xi$ and $h$ satisfy
\begin{align}
    \dddot{\xi} + [\xi, [\dot{\xi}, \xi]] + & \beta(L_{h^{-1 \ast}} \grad_1^{\text{Bi}} \tilde{V}_{\ext}(h, e)) = 0, \label{eqq2_sym_bi_comp} \\
    \dot{h}(t) &= h(t)\xi(t). \label{eqqh_sym_bi_comp}
\end{align}
\end{proposition}

\begin{proof}
Observe that $[\mathfrak{h}, \mathfrak{h}] \subset \mathfrak{s}$ implies that $\ad^\dagger_{\mathfrak{h}} \mathfrak{h} \subset \mathfrak{s}$. From \ref{lemma: covg-decomp}, it then follows that $\hcon_x y = 0$ for all $x, y \in \mathfrak{h}$. Equations \eqref{eqq2_sym} - \eqref{eqqh_sym} then follow immediately from \ref{prop: reduction_left_inv_H}. Under assumption $\textbf{S1}^{\ast}$, equations \eqref{eqq2_sym_bi} - \eqref{eqqh_sym_bi} follow upon applying \eqref{cartan} to Proposition \ref{prop: reduction_bi_inv_H}. Similarly, equations \eqref{eqq2_sym_bi_comp} - \eqref{eqqh_sym_bi_comp} follow from the same strategy used to show \eqref{eqq2_sym_bi} - \eqref{eqqh_sym_bi} hold, together with the proof strategy outlined in \ref{prop: reduction_left_inv_comp}.
\end{proof}

We now return the problem of obstacle avoidance. In particular, we consider the potential $\displaystyle{V(q) = \frac{\tau}{1 + (d_H(q, q_0)/D)^{2N}}}$, with the extended potential given by $\displaystyle{V_{\ext}(q, q_0) = \frac{\tau}{1 + (d_H(q, q_0)/D)^{2N}}}$, where $d_H: H \to H$ is the Riemannian distance on $H$ corresponding to $\left< \cdot, \cdot \right>_H$. We first show that $d_H$ is $G$-invariant, from which it follows immediately that $V_{\ext}$ is $G$-invariant. 

\begin{lemma}\label{lemma: distance-left-inv_H}
$d(gq, gp) = d(q, p)$ for all $g \in G, q, p \in H$.
\end{lemma}
\begin{proof}
Since $H$ is complete as a Riemannian manifold, there exists a geodesic $\gamma: [0, 1] \to H$ which minimizes the length functional $\displaystyle{L(c) = \int_0^1 \|\dot{c}(t)\|_H dt}$ among all smooth curves $c:[0, 1] \to H$ satisfying $\ c(0) = p, \ c(1) = q$. Moreover, we have $d(p, q) = L(\gamma)$. By Lemma \ref{lemma: G-inv_met}, we obtain $d(p, q) = L(\gamma) = L(g\gamma) \ge d(gp, gq)$, since in particular $g \gamma$ is a smooth curve such that $g\gamma(0) = gp, \ g\gamma(1) = gq$. On the other hand, there exists some geodesic $\gamma^\ast$ such that $L(\gamma^\ast) = d(gp, gq)$, and so $d(gp, gq) = L(\gamma^\ast) = L(g^{-1}\gamma^\ast) \ge d(p, q)$. It follows that $d(p, q) = d(gp, gq)$.
\end{proof}
It follows immediately from Lemma \ref{lemma: distance-left-inv_H} and $\Phi_g \circ \pi = \pi \circ L_g$ that $\displaystyle{\tilde{V}_{\ext}(g, g_0) = \frac{\tau}{1 + (d_H(\pi(g), \pi(g_0))/D)^{2N}}}$ is left-invariant. However, observe that in order to take advantage of a bi-invariant metric on $G$ to calculate the potential through the Lie exponential map (as we did in Section \ref{sec: red_biinv}), we must have $\displaystyle{\tilde{V}_{\ext}(g, g_0) = \frac{\tau}{1 + (d_G(g, g_0)/D)^{2N}}}$. Unfortunately, it is not necessarily the case $d_G(g, g_0) = d_H(\pi(g), \pi(g_0))$. In fact, if we let $c: [0, 1] \to G$ be the minimizing geodesic such that $d_G(g, g_0) = L(c)$, then $\bar{c} := \pi \circ c$ is the minimizing geodesic connecting $\pi(g)$ and $\pi(g_0)$ (here we assume that $g, g_0 \in G$ are contained in some geodesically convex ball). Hence,
\begin{align*}
    d_G(g, g_0) &= \int_0^1 \|\dot{c}\|_G dt \\
    &\ge \int_0^1 \|\Hor(\dot{c})\|_G dt \\
    &= \int_0^1 \|\dot{\bar{c}}\|_H dt \\
    &= d_H(\pi(g), \pi(g_0)).
\end{align*}
In particular, equality holds if and only if the geodesic $c$ is horizontal. The horizontal lift of a geodesic on $H$ is in turn a geodesic on $G$, so that in particular the curve $c^\ast := \widetilde{\pi \circ c}$ with $c^\ast(0) = g$ is the minimizing geodesic connecting its endpoints However, it will not generally be true that $c^\ast(1) = g_0$—we need only have $c^\ast(1) \in \pi^{-1}(\{ \pi(g_0) \}).$ This leads to the following lemma.

\begin{lemma}\label{lemma: dist_H_to_G}
Let $g \in G.$ Then, for all $h$ in a geodesically convex neighborhood of $g$, there exists a unique $h^\ast \in \pi^{-1}(\{\pi(h)\})$ such that $d_G(g, h^\ast) = d_H(\pi(g), \pi(h^\ast))$. Moreover, the function $\theta: G \times G \to G$ defined implicitly by $\theta(g, h) = h^\ast$ is smooth and left-invariant.
\end{lemma}

\begin{proof}
We have already seen that such an $h^\ast$ exists, as we need only consider the end point of the horizontal lift of the minimizing geodesic on $H$ connecting $\pi(g)$ and $\pi(h)$ with initial point $g$. The uniqueness of $h^\ast$ follows by the uniqueness of geodesics satisfying boundary conditions in a geodesically convex neighborhood. 

To see that $\theta$ is smooth, first let $\theta_g(\cdot) := \theta(g, \cdot)$ and observe that we have $d_G(g, \theta_g(h)) = d_H(\pi(g), \pi(h))$ for all $h$ in a geodesically convex neighborhood of $G$. Moreover, it is clear that we have \\$\displaystyle{\min_{r \in \pi^{-1}(\{\pi(h)\})} d_G(g, r) = d_H(\pi(g), \pi(h)) = d_G(g, h^\ast)}$. In particular, by uniqueness of the minimizing geodesic, we find that $\displaystyle{\theta_g(h) = \underset{r \in \pi^{-1}(\{\pi(h)\})}{\arg \min} d_G(g, r)}.$ Since $d_G$ is smooth and $r \mapsto d_G(g, r)$ has a unique global minimum for each $g \in G$, we find that $\theta_g$ is smooth by an application of the implicit function theorem. A similar argument then shows that $\theta$ is smooth.

Now let $\gamma_1: [0,1] \to G$ be the minimizing geodesic from $\theta(g, g_0)$ to $g_0$. Then for any $h \in G$, the curve $h \gamma_1$ satisfies $h \gamma_1(0) = hg$ and $h \gamma_1(1) = h \theta(g, g_0)$, and we have $d(g, \theta(g, g_0)) = d(hg, h \theta(g, g_0))$ by Lemma \ref{lemma: distance-left-inv}. Now let $\gamma_2: [0, 1] \to G$ be the minimizing geodesic from $hg$ to $\theta(hg, hg_0)$. By definition of $\theta$, we must have $d(hg, \theta(hg, hg_0)) \le d(hg, h \theta(g, g_0))$, and by another application of Lemma \ref{lemma: distance-left-inv}, we find that $d(g, h^{-1} \theta(hg, hg_0)) \le d(g, \theta(g, g_0))$, which is only possible if in fact they are equal. By uniqueness of minimizing curves, we then find $h \theta(g, g_0) = \theta(hg, hg_0)$. 
\end{proof}

Now consider an obstacle $q_0 \in H$, and fix some $g_0 \in \pi^{-1}(\{q_0\})$. In the context of Lemma \ref{lemma: dist_H_to_G}, we then have $\displaystyle{\tilde{V}_{\ext}(g, g_0) = \frac{\tau}{1 + (d_G(\theta(g), g_0)/D)^{2N}}}$ is smooth and left-invariant. From equation \eqref{potential_exp}, we then obtain

\begin{equation}\label{potential_exp_homo}
    L_{h(t)^{-1 \ast}}\grad_1 \tilde{V}_{\ext}(\theta(h(t)), e) = -\frac{2N\tau \|\exp^{-1}_{e}(\theta(h(t)))\|^{2N-2}}{D^{2N}(1 + (\|\exp^{-1}_{e}(\theta(h(t)))\|/D)^{2N})^2}\exp^{-1}_{e}(\theta(h(t))).
\end{equation} 
Moreover, from \eqref{potential_Log}, we obtain
\begin{equation}\label{potential_Log_homo}
    L_{h(t)^{-1 \ast}}\grad_1 \tilde{V}_{\ext}(h(t), e) = -\frac{2N\tau \|\Log(\theta(h(t)))\|^{2N-2}}{D^{2N}(1 + (\|\Log(\theta(h(t)))\|/D)^{2N})^2}\Log(\theta(h(t)))
\end{equation} 
in the case that $G$ is equipped with a bi-invariant metric, and from \eqref{potential_Log_left}, we have
\begin{equation}\label{potential_Log_left_homo}
    \beta(L_{h(t)^{-1 \ast}}\grad_1^{\text{Bi}} \tilde{V}_{\ext}(h(t), e)) = -\frac{2N\tau \|\Log(\theta(h(t)))\|^{2N-2}}{D^{2N}(1 + (\|\Log(\theta(h(t)))\|/D)^{2N})^2}\beta(\Log(\theta(h(t))))
\end{equation} 
in the case that $G$ is equipped with both left-invariant and bi-invariant metrics.
\begin{remark}\label{remark: potentials_local}
Following Corollary \ref{Cor: arbitrary_potential}, we may alternatively consider the potential \\$\displaystyle{\tilde{V}(g) = \frac{\tau}{1 + (d_G(g, g_0)/D)^{2N}}}$, whose left-invariant extension is given by $\displaystyle{\tilde{V}_{\ext}(g. g_0) = \frac{\tau}{1 + (d_G(g, g_0)/D)^{2N}}}.$ In this case, equations \eqref{ad_dag_h} - \eqref{eqqh_h} take the form:
\begin{align}
    \dot{\xi} &= \eta + \Hor(\ad^\dagger_{\xi} \xi), \label{ad_dag_h_loc}\\
    \ddot{\eta} + 2\nabla^{\mathfrak{h}}_\xi \dot{\eta} + \nabla^{\mathfrak{h}}_{\eta} \eta + \nabla^{\mathfrak{h}}_{\Hor(\ad^\dagger_\xi \xi)} \eta +& \nabla^{\mathfrak{h}}_\xi \nabla^{\mathfrak{h}}_\xi \eta + \tilde{Q}\big{(}\eta, \xi \big{)}\xi + \Hor(L_{h^{-1 \ast}} \grad_1 \tilde{V}_{\ext}(h, e)) = 0, \label{eqq2_h_loc} \\
    \dot{h}(t) &= h(t)\xi(t). \label{eqqh_h_loc}
\end{align}
The gradient potential may then be calculated exactly as in \eqref{potential_Log} if $G$ is equipped with a bi-invariant metric, or as in \eqref{potential_Log_left} if equipped with both a left-invariant and bi-invariant metric. Of course, in general the solutions to \eqref{ad_dag_h_loc}-\eqref{eqqh_h_loc} will only correspond locally to horizontal lifts of modified cubics on $H$. Moreover, the potential $V$ corresponding to these modified cubics may vary among the curve segments, and we will not in general know its exact form. 

However, for the purpose of applications, we only care that $q := g \circ \pi$ is well-behaved. In particular, that it is smooth, avoids obstacles, and is sufficiently close to a Riemannian cubic. This, of course, still holds true. Namely, if we took $\tilde{V} \equiv 0$, $q$ would be a Riemannian cubic polynomial—so that, away from obstacles, $q$ behaves as a Riemannian cubic. Moreover, we may guarantee that $q$ avoids the given obstacle by choosing $\tau, D, N$ appropriately (in particular, $\tau$ will determine how "far" from a Riemannian cubic $q$ is in a global sense). Hence we see that, morally, $q$ has all of the relevant features that we desire from modified cubics, despite not globally being one in the strict sense of the definition.
\end{remark}

\subsection{Example 3: Reduction for $S^2$}\label{Section: S^2}
Consider the sphere $S^2$ equipped with the round metric (that is, the induced metric from the embedding $S^2 \hookrightarrow \mathbb{R}^3$). Let $\{e_1, e_2, e_3\}$ be a basis for $\R^3$ such that $q_0:= e_3$ is a point-obstacle that we wish to avoid. It is well known that $S^2$ is a Riemannian symmetric space. In particular, we have $S^2 \cong \SO(3)/\text{Stab}(e_3)$, where $\SO(3)$ is equipped with the bi-invariant metric discussed in Section \ref{section: example_geo_SO3}. Note that we may identify Stab$(e_3) \cong \SO(2)$. Moreover, the projection map $\pi: \SO(3) \to S^2$ is given by $\pi(R) = Re_3$ for all $R \in \SO(3)$. 

Following Remark \ref{remark: potentials_local}, we first consider the artificial potential $\tilde{V}: \SO(3)^2 \to \R$ defined by $\tilde{V}(g) = \frac{\tau}{1 + (d_{\SO(3)}(g, g_0)/D)^{2N}}$ some $\tau, D > 0$ and $N \in \N$, where $R_0 \in \pi^{-1}(\{e_3\})$. If $R$ satisfies \eqref{lift_cubic}, then from Proposition \ref{prop: reduction_sym_left}, Lemma \ref{Log}, and equation \eqref{potential_Log_homo}, it follows that $\hat{\Omega} := R^T \dot{R}$ and $H := R_0^T R$ must satisfy
\begin{align}
    \dddot{\Omega} + \Omega \times (\dot{\Omega} \times \Omega) + &\frac{2N\tau \phi(H)^{2N-1}}{\sin(\phi(H))D^{2N}(1 + (\phi(H)/D)^{2N})^2}(H - H^T)^\vee = 0, \label{eqq2_sym_bi_S2} \\
    \dot{H} &= H\hat{\Omega}, \label{eqqh_sym_bi_S2}
\end{align}
where $\phi(H) := \arccos(\frac12(\tr(H) - 1))$. Moreover, we have the reconstruction equation $R = R_0 H$, from which $q$ may be determined via $q = Re_3$.

\section{Conclusions}\label{sec:conclusions}
Along this paper, we have considered the variational obstacle avoidance problem on Lie groups and Riemannian homogeneous spaces. In both cases, reduced necessary conditions for optimality were derived. A number of special cases were examined in which the Riemannian distance—and in turn the obstacle avoidance potential—can be calculated explicitly. In the case of Riemannian homogeneous spaces, we consider a new variational problem written in terms of a connection on the horizontal bundle of the underlying Lie group. Through this formalism, we are able to derive the necessary conditions for optimality and reduce them with a symmetry breaking potential. Applications to obstacle avoidance for rigid bodies on $\SO(3)$ and $S^2$ were also considered. 

For future work, we would like to consider the inverse problem for Riemannian homogeneous spaces. That is, to understand the conditions under which solutions to the variational principle will be horizontal (and thus correspond to the horizontal lifts of solutions in the homogeneous space). We would also like to consider the design of potentials that may be used in the case where the Lie group only admits left-invariant metrics, or how approximate solutions can be generated to increase the utility of variational obstacle avoidance in path-planning strategies. 
%\section*{Acknowledgments}

\bibliographystyle{siamplain}
\bibliography{references}

\end{document}

%% file: ex_shared.tex
\usepackage{lipsum}
\usepackage{amsfonts}
\usepackage{graphicx}
\usepackage{epstopdf}
\usepackage{algorithmic}
\ifpdf
  \DeclareGraphicsExtensions{.eps,.pdf,.png,.jpg}
\else
  \DeclareGraphicsExtensions{.eps}
\fi

\usepackage{enumitem}
\setlist[enumerate]{leftmargin=.5in}
\setlist[itemize]{leftmargin=.5in}

\newsiamremark{remark}{Remark}
\newsiamremark{hypothesis}{Hypothesis}
\crefname{hypothesis}{Hypothesis}{Hypotheses}
\newsiamthm{claim}{Claim}

\headers{Reduction by Symmetry in Obstacle Avoidance Problems}{J. R. Goodman, and L. J.  Colombo}

\title{Reduction by Symmetry in Obstacle Avoidance Problems on Riemannian Manifolds\thanks{Submitted to the editors DATE.
\funding{The project that gave rise to these results received the support of a fellowship from ”la Caixa” Foundation (ID 100010434). The fellowship code is LCF/BQ/DI19/11730028. Additionally, support has been given by the ``Severo Ochoa Programme for Centres of Excellence''in R$\&$D (CEX2019-000904-S). The authors acknowledge financial support from Grant PID2019-106715GB-C21 funded by MCIN/AEI/ 10.13039/501100011033.}}}

\author{Jacob R. Goodman\thanks{Instituto de Ciencias Matematicas
(CSIC-UAM-UC3M-UCM), Calle Nicolas Cabrera 13-15, 28049, Madrid, Spain 
  (\email{jacob.goodman@icmat.es}).}\and Leonardo J. Colombo\thanks{Centre for Automation and Robotics (CSIC-UPM), 28500
Madrid, Spain 
  (\email{leonardo.colombo@car.upm-csic.es}).}}

\usepackage{amsopn}

\makeatletter
\newcommand*{\addFileDependency}[1]{% argument=file name and extension
  \typeout{(#1)}% latexmk will find this if $recorder=0 (however, in that case, it will ignore #1 if it is a .aux or .pdf file etc and it exists! if it doesn't exist, it will appear in the list of dependents regardless)
  \@addtofilelist{#1}% if you want it to appear in \listfiles, not really necessary and latexmk doesn't use this
  \IfFileExists{#1}{}{\typeout{No file #1.}}% latexmk will find this message if #1 doesn't exist (yet)
}
\makeatother

\newcommand*{\myexternaldocument}[1]{%
    \externaldocument{#1}%
    \addFileDependency{#1.tex}%
    \addFileDependency{#1.aux}%
}